\documentclass{amsart}

\usepackage[
	pdftitle={Topological recursion for irregular spectral curves},
	pdfauthor={Norman Do and Paul Norbury},
	ocgcolorlinks,
	linkcolor=linkblue,
	citecolor=linkred,
	urlcolor=linkred]
{hyperref}

\usepackage{amsfonts}
\usepackage{amssymb}
\usepackage{bbm}
\usepackage{bm}
\usepackage{booktabs}
\usepackage[hang,flushmargin]{footmisc} 
\usepackage[left=25mm,right=25mm,top=25mm,bottom=25mm]{geometry}
\usepackage{graphicx}
\usepackage{mathpazo}
\usepackage{microtype}
\usepackage{multicol}
\usepackage{tikz}
     \usetikzlibrary{arrows,shapes}
\usepackage{xcolor}
     \definecolor{linkred}{rgb}{0.6,0,0}
     \definecolor{linkblue}{rgb}{0,0,0.6}

\theoremstyle{plain}
     \newtheorem{theorem}{Theorem}
     
     \newtheorem{proposition}{Proposition}[section]
     \newtheorem{conjecture}{Conjecture}
     \newtheorem{corollary}[proposition]{Corollary}

\theoremstyle{definition}
     \newtheorem{example}[proposition]{Example}
     \newtheorem{definition}[proposition]{Definition}
     \newtheorem{remark}[proposition]{Remark}

\newcommand{\nb}{N\hspace{-1.35mm}B}

\newcommand{\res}{\mathop{\mathrm{Res}}}
\newcommand{\bc}{\mathbb{C}}
\newcommand{\bp}{\mathbb{P}}
\newcommand{\bz}{\mathbb{Z}}
\newcommand{\cb}{\mathcal{B}}
\newcommand{\cn}{\mathcal{N}}
\newcommand{\f}{\mathcal{F}}
\newcommand{\cl}{\mathcal{L}}
\newcommand{\vc}{\mathcal{V}}
\newcommand{\modm}{\mathcal{M}}
\newcommand{\fat}{\f_{g,n}}

\newcommand {\dd}{\mathrm{d}}
\newcommand {\h}{\hbar}
\newcommand {\x}{\widehat{x}}
\newcommand {\xx}{\bm{x}}
\newcommand {\zz}{\bm{z}}
\newcommand {\y}{\widehat{y}}
\newcommand {\Z}{\overline{Z}}
\newcommand{\mmu}{\boldsymbol{\mu}}
\newcommand{\LL}{\boldsymbol{L}}

\DeclareRobustCommand{\stirling}{\genfrac{[}{]}{0pt}{}}

\numberwithin{equation}{section}

\begin{document}

\title{Topological recursion for irregular spectral curves}
\author{Norman Do \and Paul Norbury}
\address{School of Mathematical Sciences, Monash University, VIC 3800 Australia}
\address{Department of Mathematics and Statistics, University of Melbourne, VIC 3010 Australia}
\email{\href{mailto:norm.do@monash.edu}{norm.do@monash.edu}, \href{mailto:pnorbury@ms.unimelb.edu.au}{pnorbury@ms.unimelb.edu.au}}
\thanks{The authors were partially supported by the Australian Research Council grants DE130100650 (ND) and DP1094328 (PN)}
\date{\today}

\begin{abstract}
We study topological recursion on the irregular spectral curve $xy^2-xy+1=0$, which produces a weighted count of dessins d'enfant. This analysis is then applied to topological recursion on the spectral curve $xy^2=1$, which takes the place of the Airy curve $x=y^2$ to describe asymptotic behaviour of enumerative problems associated to irregular spectral curves. In particular, we calculate all one-point invariants of the spectral curve $xy^2=1$ via a new three-term recursion for the number of dessins d'enfant with one face.
\end{abstract}

\maketitle

\setlength{\parskip}{0pt}
\tableofcontents
\setlength{\parskip}{6pt}

\section{Introduction}  \label{sec:intro}

{\em Topological recursion} developed by Eynard, Orantin and Chekhov produces invariants of a Riemann surface $C$ equipped with two meromorphic functions $x, y: C\to \mathbb{C}$ and a bidifferential $B(p_1,p_2)$ for $p_1, p_2 \in C$~\cite{CEyHer,EOrInv}. We require the zeros of $\dd x$ to be simple and refer to the data $(C,B,x,y)$ as a {\em spectral curve}. For integers $g \geq 0$ and $n \geq 1$, the invariant $\omega^g_{n}$ is a multidifferential on $C$ or, in other words, a tensor product of meromorphic 1-forms on $C^n$. In this paper, all spectral curves will have underlying Riemann surface $\mathbb{CP}^1$ and bidifferential $B = \frac{\dd z_1 \otimes \dd z_2}{(z_1 - z_2)^2}$. In that case, we may specify the spectral curve parametrically via the meromorphic functions $x(z)$ and $y(z)$. We call a spectral curve {\em regular} if it is non-singular at the zeros of $\dd x$~---~for example, if the curve is non-singular. See Section~\ref{sec:EO} for precise definitions.

The invariants $\omega^g_n$ of the {\em Airy curve} $x=y^2$ are (total derivatives of) the following generating functions for intersection numbers of Chern classes of the tautological line bundles $\cl_i$ on the moduli space of stable curves $\overline{\modm}_{g,n}$ \cite{EOrTop}.
\begin{equation}  \label{airy}
K_{g,n}(z_1, \ldots, z_n) = \frac{1}{2^{2g-2+n}}\sum_{|\mathbf{d}| = 3g-3+n} \int_{\overline{\modm}_{g,n}} c_1(\cl_1)^{d_1} \cdots c_1(\cl_n)^{d_n}\prod_{i=1}^n\frac{(2d_i-1)!!}{z_i^{2d_i+1}}
\end{equation}

A regular spectral curve locally resembles the Airy curve $x=y^2$ near zeros of $\dd x$, which are assumed to be simple. This leads to universality in the behaviour of topological recursion on regular spectral curves~---~the invariants are related to intersection theory on $\overline{\modm}_{g,n}$. Three progressively more refined statements of this relationship are as follows.
\begin{enumerate}
\item Eynard and Orantin \cite{EOrTop} proved that the invariants $\omega^g_n$ behave asymptotically near a regular zero of $\dd x$ like the invariants $\omega^g_n$ of the Airy curve $x=y^2$ at the origin. Hence, they store the intersection numbers appearing in equation~\eqref{airy}.
\item Eynard \cite{EynInv,EynInt} pushed this further, proving that the lower order asymptotic terms of $\omega^g_n$ on a regular spectral curve also encode intersection numbers. These come in the form of explicit combinations of Hodge integrals on $\overline{\modm}_{g,n}$ and a generalisation $\overline{\modm}^a_{g,n}$, which he calls the moduli space of $a$-coloured stable curves.
\item For a special class of regular spectral curves, Dunin--Barkowski, Orantin, Shadrin and Spitz~\cite{DOSSIde} extended the results of Eynard, by proving that the multidifferentials $\omega^g_n$ encode 
ancestor invariants in a cohomological field theory, which is fundamentally related to intersection theory on $\overline{\modm}_{g,n}$.
\end{enumerate}

In this paper, we consider irregular spectral curves that locally resemble the curve $xy^2=1$ near some zeros of $\dd x$. (In Section~\ref{sec:irr}, we show that any other local irregular behaviour is ill-behaved.) For such curves, the local behaviour of the invariants $\omega^g_n$ is no longer determined by the intersection numbers of equation~\eqref{airy}. An analogue of statement (1) above holds, although we do not currently have an analogue of equation~\eqref{airy} to relate the invariants of the spectral curve $xy^2=1$ to a moduli space. Instead, we consider a specific problem~---~the enumeration of dessins d'enfant~---~which is governed by an irregular spectral curve. We rely on this concrete example to shed light on the local behaviour of all irregular spectral curves. To achieve this, we first show that the enumeration of dessins d'enfant satisfies topological recursion on the irregular spectral curve $xy^2-xy+1=0$. We then prove a three-term recursion for its 1-point invariants, and use this to determine an exact formula for the 1-point invariants of the spectral curve $xy^2=1$.

A {\em dessin d'enfant} is a bicoloured graph embedded in a connected orientable surface, such that the complement is a union of disks. The term {\em bicoloured} means that the vertices are coloured black and white such that each edge is adjacent to one vertex of each colour. Consequently, the underlying graph of a dessin d'enfant is necessarily bipartite. One can interpret a dessin d'enfant as a branched cover $\pi:\Sigma\to\bp^1$ unramified over $\bp^1-\{0,1,\infty\}$, often referred to as a {\em Belyi map}. The bicoloured graph is given by $\pi^{-1}([0,1])\subset\Sigma$, with the points $\pi^{-1}(\{0\})$ representing black vertices and the points $\pi^{-1}(\{1\})$ representing white vertices. 

% add a reference for dessin's d'enfant?

Let $\cb_{g,n}(\mu_1, \ldots, \mu_n)$ be the set of all genus $g$ Belyi maps $\pi:\Sigma\to\bp^1$ with ramification divisor over $\infty$ given by $\pi^{-1}(\infty)=\mu_1p_1+ \cdots + \mu_np_n$, where the points over $\infty$ are labelled $p_1, \ldots, p_n$. Two Belyi maps $\pi_1:\Sigma_1\to\bp^1$ and $\pi_2:\Sigma_2\to\bp^1$ are isomorphic if there exists a homeomorphism $f:\Sigma_1\to\Sigma_2$ that covers the identity on $\bp^1$ and preserves the labelling over $\infty$. Equivalently, one can interpret $\cb_{g,n}(\mu_1, \ldots, \mu_n)$ as the set of connected genus $g$ dessins d'enfant with $n$ labelled boundary components of lengths $2\mu_1, \ldots, 2\mu_n$. By a boundary component of a dessin d'enfant, we mean a cycle in the underlying graph corresponding to the boundary of one of the labelled disks in $\Sigma$. We require an isomorphism between two dessins d'enfant to preserve the labelling on their boundary components.

\begin{definition} 
For any $\mmu=(\mu_1,\ldots,\mu_n)\in\bz_+^n$, define
\[
B_{g,n}(\mu_1,\ldots,\mu_n)=\sum_{\Gamma\in\cb_{g,n}(\mmu)}\frac{1}{|{\rm Aut\ }\Gamma|},
\]
where ${\rm Aut\ }\Gamma$ denotes the automorphism group of the dessin d'enfant $\Gamma$.
\end{definition}

For integers $g \geq 0$ and $n \geq 1$, define the generating function
\begin{equation}  \label{genfn}
F_{g,n}(x_1,\ldots,x_n) = \sum_{\mu_1, \ldots, \mu_n=1}^\infty B_{g,n}(\mu_1,\ldots,\mu_n)\prod_{i=1}^n x_i^{-\mu_i}.
\end{equation}
If we let $x=z+\frac{1}{z}+2$ and define $x_i=x(z_i)$ for $i = 1, 2, \ldots, n$, then we may observe that this generating function is well-behaved with respect to $z_1, \ldots, z_n$. In particular, Theorem~\ref{th:main} below implies that $F_{g,n}(x_1, \ldots, x_n)$ is a rational function of $z_1, \ldots, z_n$ for $2g-2+n>0$, with poles only at $z_i = \pm 1$ of certain orders. The derivative $y = \frac{\partial}{\partial x} F_{0,1}(x)$ is also rational and together with $x$ defines a plane curve known as the spectral curve. More precisely, Theorem~\ref{th:main} shows that the (total derivatives of) $F_{g,n}(x_1, \ldots, x_n)$ satisfy topological recursion on the spectral curve $xy^2 - xy + 1 = 0$, given parametrically by
\begin{equation}  \label{specurve}
x=z+\frac{1}{z}+2 \qquad \text{and} \qquad y=\frac{z}{1+z}.
\end{equation}
Furthermore, the topological recursion determines the generating functions $F_{g,n}(x_1, \ldots, x_n)$ uniquely.

\begin{theorem}  \label{th:main}
For $2g-2+n > 0$, the multidifferential
\begin{equation} \label{}
\Omega_{g,n}(z_1, \ldots ,z_n) = \frac{\partial}{\partial x_1} \cdots \frac{\partial}{\partial x_n} F_{g,n}(x_1, \ldots, x_n)~\dd x_1 \otimes \cdots \otimes \dd x_n
\end{equation}
is the analytic expansion of the invariant $\omega^g_n$ of the spectral curve \eqref{specurve} at the point $x_1 = \cdots = x_n = \infty$.
\end{theorem}

From the universality property described above, the asymptotic behaviour of the generating function \eqref{genfn} near its pole at $(z_1,\ldots,z_n)=(1,\ldots,1)$ is given by $4^{2g-2+n}K_{g,n}(z_1, \ldots, z_n)$. More precisely, we have
\[
F_{g,n}(x_1, \ldots, x_n)=2^{2g-2+n}\sum_{|\bm{d}| = 3g-3+n} \int_{\overline{\modm}_{g,n}}c_1(\cl_1)^{d_1} \cdots c_1(\cl_n)^{d_n}\prod_{i=1}^n\frac{(2d_i-1)!!}{(z_i-1)^{2d_i}} + [\,\text{lower order poles}\,].
\] 
Note that the spectral curve given by equation~\eqref{specurve} is irregular. The local behaviour near its poles $(z_1, \ldots, z_n)=(\pm 1, \ldots, \pm 1)$ is our main interest. One immediate consequence of irregularity is the novel feature that the genus 0 generating functions $F_{0,n}(x_1, \ldots, x_n)$ are analytic at $(z_1, \ldots, z_n)=(-1, \ldots, -1)$. More generally, the orders of poles of $F_{g,n}(x_1, \ldots, x_n)$ at $(z_1, \ldots, z_n)=(-1, \ldots, -1)$ are independent of $n$. This is in contrast to $\omega_n^g$ having poles of order $6g-4+2n$, which is the case for most of the spectral curves that appear in the literature.

Properties of rational functions on the curve \eqref{specurve} yield a structure theorem for $B_{g,n}$ --- see Theorem~\ref{th:struc} --- as well as explicit formulae. For example, we have
\[
B_{0,n}(\mu_1, \ldots, \mu_n)=\frac{2^{1-n}(n-1)!}{|\mmu|\,(|\mmu|+1)}\binom{|\mmu|+1}{n-1}\prod_{i=1}^n\binom{2\mu_i}{\mu_i},\quad \text{where } |\mmu|=\sum_{i=1}^n \mu_i.
\]
Another consequence of Theorem~\ref{th:main} is a general property of the invariants $\omega^g_n$, known as the {\em dilaton equation}. For the spectral curve of interest, it implies that
\[
B_{g,n+1}(1,\mu_1, \ldots, \mu_n)-B_{g,n+1}(0,\mu_1, \ldots, \mu_n)=\frac{1}{2}(|\mmu|+2g-2+n) \, B_{g,n}(\mu_1, \ldots, \mu_n),
\]
where one can make sense of evaluation at $\mu_i=0$ using the structure theorem for $B_{g,n}$. Moreover, we prove in Proposition~\ref{th:pointed} that for $n$ positive, $B_{g,n+m}(\mu_1, \ldots, \mu_n, 0, 0, \ldots, 0)$ has a combinatorial meaning~---~it enumerates dessins d'enfant with $m$ black vertices labelled.

The enumerative problem in this paper and the associated spectral curve given by equation~\eqref{specurve} are closely related to others in the literature. Topological recursion on rational spectral curves with $x=\alpha+\gamma(z+\frac{1}{z})$ describe enumeration of discrete surfaces \cite{EOrTop}, which includes the special case of lattice points in moduli spaces of curves~\cite{NorStr}, a more refined version of dessin enumeration~\cite{AChMat,DiFRec,KZoVir}, and the Gromov--Witten invariants of $\bp^1$~\cite{DOSSIde,NScGro}. However, note that each of these examples is governed by a regular spectral curve. 

A {\em quantum curve} of a spectral curve $P(x,y)=0$ is a Schr\"odinger-type equation $\widehat{P}(\x, \y) \, Z(x,\h)=0$, where $\widehat{P}(\x, \y)$ is a non-commutative quantisation of the spectral curve with $\x=x$ and $\y=\h\frac{\partial}{\partial x}$. This differential operator annihilates a wave function $Z(x,\h)$, which is a formal series in $\hbar$ associated to the spectral curve. The path from the quantum curve to the spectral curve is well-defined~---~in the semi-classical limit $\hbar\to 0$, the differential operator reduces to a multiplication operator that vanishes precisely on the spectral curve. On the other hand, constructing the quantum curve from the spectral curve is not canonical. The main issues lie in the construction of the wave function and the ambiguity in ordering the non-commuting operators $\x$ and $\y$. One remedy for these issues is a conjectural construction of the wave function $Z(x, \h)$ from the invariants $\omega^g_n$ of the spectral curve, suggested for example by Gukov and Su{\l}kowski~\cite{GSuAPo}. We prove this conjecture for the spectral curve $xy^2-xy+1=0$.

Define the wave function as follows.
\[
Z(x, \h) = x^{-1/\hbar}\exp \bigg[ \sum_{g=0}^\infty \sum_{n=1}^\infty \frac{\h^{2g-2+n}}{n!} \,F_{g,n}(x, x, \ldots, x) \bigg]
\]

\begin{theorem} \label{thm:qua}
The quantum curve of $xy^2-xy+1=0$ is given by $(\y \x \y -\y\x + 1) \, Z(x,\h)=0$.
\end{theorem}

Strictly speaking, to make sense of the action of a differential operator on a formal series in $\hbar$, it is necessary to know that all sums are finite. In Section~\ref{sec:qua}, we give a more precise statement of Theorem~\ref{thm:qua}, in terms of a differential operator annihilating the formal series $\Z(x, \h) = x^{1/\h} \, Z(x, \h) \in \mathbb{Q}[\h^{\pm 1}][[x^{-1}]]$.\\

One of the main purposes of the present paper is to understand the universality exhibited by the invariants associated to the irregular spectral curve
\begin{equation}  \label{air}
xy^2=1.
\end{equation}
At irregular zeros of $\dd x$, this plays the role of the Airy curve at regular zeros of $\dd x$. We expect many results relating the Airy curve to invariants of spectral curves to have analogues in this setting. In particular, we expect the invariants of our curve to be related to a new moduli space. In Section~\ref{sec:asym}, we apply topological recursion directly to the spectral curve given by equation~\eqref{air}. We calculate putative volumes of these unidentified moduli spaces and dually, intersection numbers on them. The invariants are non-zero only for positive genus, much like enumeration of branched covers of a torus or volumes of spaces of holomorphic differentials.

Indirectly, we use the asymptotic behaviour of $F_{g,n}(x_1, \ldots, x_n)$ defined in equation~\eqref{genfn} near its pole $(z_1, \ldots, z_n)=(-1, \ldots, -1)$ to study the spectral curve~\eqref{air}. More generally, the type of enumerative problem governed by \eqref{air} necessarily has no contribution in genus 0. To make this idea clearer, consider for the moment the enumeration of non-bipartite fatgraphs with bipartite boundary components --- in other words, boundary components of even lengths. For example, the square graph pictured below left is bipartite and may be considered the boundary of the non-bipartite genus 1 fatgraph pictured below right.

\begin{center}
\begin{tikzpicture}[scale=0.5]
\filldraw (0,0) circle(2mm);
\filldraw (0,2) circle(2mm);
\filldraw (2,2) circle(2mm);
\filldraw (2,0) circle(2mm);
\draw[ draw=blue!10!black]
(0,0)--(2,0)--(2,2)--(0,2)--cycle;
\end{tikzpicture}
\quad\quad\quad\quad
\begin{tikzpicture}[scale=0.5]
\draw (0,0) circle(10mm);
\draw (1,1) circle(10mm);
\filldraw (0,1) circle(2mm);
\end{tikzpicture}
\end{center}

Define $\nb_g(\mu_1, \ldots, \mu_n)$ to be the weighted count of connected genus $g$ non-bipartite fatgraphs with labelled bipartite boundaries of lengths $2\mu_1, \ldots, 2\mu_n$ and such that the vertices are required to have valency greater than or equal to two~---~see equation~\eqref{nonbip}. The valency condition on the vertices reduces the growth in $\mu_i$ from exponential to polynomial. In particular, this invariant vanishes in genus zero since bipartite boundary components implies that the graph is bipartite for simple homological reasons. One can show that $\nb_g(\mu_1, \ldots, \mu_n)$ is quasi-polynomial in $(\mu_1, \ldots, \mu_n)$ modulo 2. An interesting invariant is obtained by measuring its failure to be polynomial, which we do in the following way. Recall that if $p(\mu)$ is a quasi-polynomial modulo 2, then it has a natural decomposition $p(\mu)=p^+(\mu) + (-1)^\mu p^-(\mu)$, where $p^{\pm}(\mu)$ are polynomials. 
For the analogous decomposition of a quasi-polynomial $p(\mu_1, \ldots, \mu_n)$ in severable variables, it is the coefficient of $(-1)^{|\mmu|}$ that interests us.
\[
p(\mu_1, \ldots, \mu_n)=p^+(\mu_1, \ldots, \mu_n)+(-1)^{|\mmu|}p^-(\mu_1, \ldots, \mu_n) + [\, \text{other terms involving powers of $-1$} \,]
\]
For example, 
\[
\nb_0(\mu_1, \ldots, \mu_n)=0,\qquad\nb_1(\mu_1)=\frac{1}{8}\mu_1^2-\frac{\epsilon(\mu_1)}{8}
,\qquad \nb_1(\mu_1,\mu_2)=\frac{1}{16} (\mu_1^2 + \mu_2^2 ) (\mu_1^2 + \mu_2^2 -2 )-\frac{\epsilon(|\mmu|)}{16},
\]
%\[ \nb_2(b)=\frac{1}{18432}(b^2-1)^2(b^2-9)(5b^2-19)+\epsilon(b)\frac{6b^2-19}{2048} \]
where $\epsilon(\mu)=\frac{1}{2}[1-(-1)^{\mu}]$. From these expressions, one can extract the non-polynomial parts $p^-_1(\mu_1)=\frac{1}{8}$ and $p^-_1(\mu_1,\mu_2)=\frac{1}{16}$, which determine the invariants $\omega^1_1$ and $\omega^1_2$ of the spectral curve $xy^2=1$. %, $p^-_2(b)=\frac{6b^2-19}{2048}$. 
More generally, the top degree part of $p_g^-(\mu_1, \ldots, \mu_n)$ is equivalent to the invariant $\omega^g_{n}$ of $xy^2=1$.

Returning to the enumeration of dessins d'enfant, consider the three-term recursion
\[
n(n+1) \,B_{g,1}(n)=2(2n-1)(n-1) \, B_{g,1}(n-1)+(n-1)^2(n-2)^2 \,B_{g,1}(n-2),
\]

which is proven in Section~\ref{sec:3term}. We remark that it has a rather different character to the topological recursion of Theorem~\ref{th:main}. In particular, it enables one to calculate $B_{g,1}$ recursively from $B_{h,1}$ for $h \leq g$, without requiring $B_{h,n}$ for $n \geq 2$. It is analogous to the three-term recursion for fatgraphs with one face of Harer and Zagier~\cite{HZaEul}. Our three-term recursion implies a recursion satisfied by the 1-point invariants of the spectral curve $xy^2=1$, which then leads to the following exact formula.

\begin{theorem}  \label{th:1point}
The 1-point invariants of the spectral curve $xy^2=1$, given parametrically by $x(z) = z^2$ and $y(z) = \frac{1}{z}$, are
\[
\omega^g_1(z)=2^{1-8g} \, \frac{(2g)!^3}{g!^4(2g-1)} \, z^{-2g} \, \dd z.
\]
\end{theorem}

We would hope to recognise some type of intersection number in the formula of Theorem~\ref{th:1point} analogous to the 1-point invariants of the Airy curve $x=y^2$, which are given by
\[
\omega^{g}_1(z)_{\text{Airy}} = 2^{3-8g}\, \frac{(6g-3)!}{3^gg!(3g-2)!} \, z^{2-6g} \, \dd z = 2^{1-2g} \, (6g-3)!! \, \int_{\overline{\modm}_{g,1}} c_1(\cl_1)^{3g-2} \, z^{2-6g} \, \dd z.
\]

The spectral curve $xy^2-xy+1=0$ bears a clear resemblance to the regular spectral curve $y^2-xy+1=0$, which has been well-studied in the literature, since it arises from a matrix model with pure Gaussian potential. The invariants of the curve $y^2-xy+1=0$ take the form $\omega^g_n = \sum M_{g,n}(\mu_1, \ldots, \mu_n) \, x_1^{-\mu_1} \cdots x_n^{-\mu_n}$, where 
\[
M_{g,n}(\mu_1, \ldots, \mu_n)=\sum_{\Gamma\in\fat(\mmu)}\frac{1}{|{\rm Aut\ }\Gamma|}.
\]
Here, $\fat(\mmu)$ denotes the set of connected genus $g$ fatgraphs~---~graphs embedded in a connected orientable surface such that the complement is a union of disks~---~with $n$ labelled boundary components of lengths $\mu_1, \ldots, \mu_n$. Again, we require an isomorphism between two fatgraphs to preserve the labelling on their boundary components. Note that each bipartite fatgraph can be bicoloured in two distinct ways, thereby producing two dessins d'enfant, so we obtain
\[
B_{g,n}(\mu_1, \ldots, \mu_n) \leq 2 M_{g,n}(2\mu_1, \ldots, 2\mu_n).
\]
In fact, equality occurs when $g=0$, although no such explicit relation exists in higher genus.

One can obtain the spectral curves $y^2-xy+1=0$ and $xy^2-xy+1=0$ from the Stieltjes transforms
\[
y=\int_{-2}^2\frac{\rho(t)}{x-t} \, \dd t = \sum_{n=0}^{\infty} \frac{C_n}{x^{2n+1}}\qquad \text{and} \qquad y=\int_0^4\frac{\lambda(t)}{x-t} \, \dd t = \sum_{n=0}^{\infty}\frac{C_n}{x^{n+1}}
\]
of the probability densities
\[
\rho(t)=\frac{1}{2\pi}\sqrt{4-t^2}\cdot\mathbbm{1}_{[-2,2]}\ \qquad \text{and} \qquad \lambda(t)=\frac{1}{2\pi}\sqrt{\frac{4-t}{t}}\cdot\mathbbm{1}_{[0,4]}.
\]
These are known as the Wigner semicircle distribution and the Marchenko--Pastur distribution, respectively. It is elementary to show that the Catalan numbers $C_n=\frac{1}{n+1}\binom{2n}{n}$ arise as moments of these probability densities.
\[
C_n=\int_{-2}^2t^{2n}\rho(t) \, \dd t = \int_0^4t^n\lambda(t) \, \dd t
\]
Regular behaviour of spectral curves arise from so-called {\em soft edge} statistics, while the irregular behaviour of the spectral curve given by equation~\eqref{specurve} arises from so-called {\em hard edge} statistics. These terms refer to the behaviour of the associated probability densities at the endpoints of the interval of support.

\section{Topological recursion}  \label{sec:EO}

Topological recursion takes as input a \emph{spectral curve} $(C,B,x,y)$ consisting of a compact Riemann surface $C$, a bidifferential $B$ on $C$, and meromorphic functions $x,y:C\to\bc$. We furthermore require that the zeros of $\dd x$ are simple and disjoint from the zeros of $\dd y$~\cite{EOrInv}. A more general setup allows {\em local spectral curves}, in which $C$ is an open subset of a compact Riemann surface. In this paper, we deal exclusively with the case when $C$ is the Riemann sphere $\mathbb{CP}^1$, with global rational parameter $z$, endowed with the bidifferential $B= \frac{\dd z_1 \otimes \dd z_2}{(z_1-z_2)^2}$. The two main examples that we consider take the pair of meromorphic functions to be $(x,y)=(z+\frac{1}{z}+2, \frac{z}{1+z})$ and $(x,y)=(z^2,\frac{1}{z})$. They are mild variants of the usual setup, since in both cases, $\dd y$ has a pole at a zero of $\dd x$. Nevertheless, as we will see below, topological recursion is well-defined in this case and retains many of the desired properties, while losing some others.

For integers $g\geq 0$ and $n \geq 1$, topological recursion outputs multidifferentials $\omega^g_n(p_1, \ldots, p_n)$ on $C$ --- in other words, a tensor product of meromorphic 1-forms on the product $C^n$, where $p_i\in C$. When $2g-2+n>0$, $\omega^g_n(p_1, \ldots, p_n)$ is defined recursively in terms of local information around the poles of $\omega^{g'}_{n'}(p_1, \ldots, p_{n'})$ for $2g'+2-n' < 2g-2+n$.

Since each zero $\alpha$ of $\dd x$ is assumed to be simple, for any point $p\in C$ close to $\alpha$, there is a unique point $\hat{p}\neq p$ close to $\alpha$ such that $x(\hat{p})=x(p)$. The recursive definition of $\omega^g_n(p_1, \ldots, p_n)$ uses only local information around zeros of $\dd x$ and makes use of the well-defined map $p\mapsto\hat{p}$ there. The invariants are defined as follows, for $C = \mathbb{CP}^1$ and $B= \frac{\dd z_1 \otimes \dd z_2}{(z_1-z_2)^2}$ with $z_i=z(p_i)$. Start with the base cases
\[
\omega^0_1=-y(z)\,\dd x(z) \qquad \text{and} \qquad \omega^0_2=\frac{\dd z_1 \otimes \dd z_2}{(z_1-z_2)^2}.
\]
For $2g-2+n>0$ and $S = \{2, \ldots, n\}$, define
\begin{equation}  \label{EOrec}
\omega^g_{n}(z_1,\zz_{S})=\sum_{\alpha}\res_{z=\alpha}K(z_1,z) \bigg[\omega^{g-1}_{n+1}(z,\hat{z},\zz_{S})+ \mathop{\sum_{g_1+g_2=g}}_{I\sqcup J=S}^\circ \omega^{g_1}_{|I|+1}(z,\zz_I) \, \omega^{g_2}_{|J|+1}(\hat{z},\zz_J) \bigg],
\end{equation}
where the outer summation is over the zeros $\alpha$ of $\dd x$ and the $\circ$ over the inner summation means that we exclude terms that involve $\omega_1^0$. We define $K$ by the following formula
\[
K(z_1,z)=\frac{-\int^z_{\hat{z}}\omega_2^0(z_1,z')}{2[y(z)-y(\hat{z})] \, \dd x(z)}=\frac{1}{2[y(\hat{z})-y(z)] \, x'(z)}\left( \frac{1}{z-z_1}- \frac{1}{\hat{z}-z_1}\right)\frac{\dd z_1}{\dd z},
\] 
which is well-defined in the vicinity of each zero of $\dd x$. Note that the quotient of a differential by the differential $\dd x(z)$ is a meromorphic function. The recursion is well-defined even when $y$ has poles at the zeros $\alpha$ of $\dd x$. It does not use all of the information in the pair $(x,y)$, but depends only on the meromorphic differential $y\,\dd x$ and the local involutions $p\mapsto\hat{p}$. For $2g-2+n>0$, the multidifferential $\omega^g_n$ is symmetric, with poles only at the zeros of $\dd x$ and vanishing residues.

In the case $(x,y)=(z^2,\frac{1}{z})$, the differential $y\,\dd x=2\,\dd z$ is analytic and non-vanishing at the zero $z=0$ of $\dd x$. This leads to the vanishing of its genus zero invariants, since the kernel $K(z_1, z)$ has no pole at $z=0$ and by induction, $\omega^0_n$ has no pole at $z=0$. Interesting invariants arise via $\omega^1_1$, since a pole at $z=0$ occurs in the expression $\omega_2^0(z, \hat{z}) = \frac{\dd z \otimes \dd \hat{z}}{(z-\hat{z})^2}$. Non-triviality of $\omega^1_1$ leads to non-triviality of $\omega^g_n$ for all $g \geq 1$.

For $2g-2+n>0$, the invariants $\omega^g_n$ of regular spectral curves satisfy the following {\em string equations} for $m = 0, 1$~\cite{EOrInv}.
\begin{equation}  \label{eq:string}
\sum_{\alpha} \res_{z=\alpha} x^my\omega^g_{n+1}(z,z_S)=-\sum_{j=1}^ndz_j\frac{\partial}{\partial z_j}\left(\frac{x^m(z_j)\omega^g_n(z_S)}{dx(z_j)}\right)
\end{equation}
They also satisfy the dilaton equation~\cite{EOrInv}
\begin{equation} \label{dilaton}
\sum_{\alpha}\res_{z=\alpha}\Phi(z)\, \omega^g_{n+1}(z,z_1, \ldots ,z_n)=(2-2g-n) \,\omega^g_n(z_1, \ldots, z_n),
\end{equation}
where the summation is over the zeros $\alpha$ of $\dd x$ and $\Phi(z)=\int^z y\,\dd x(z')$ is an arbitrary antiderivative. The dilaton equation enables the definition of the so-called {\em symplectic invariants}
\[
F_g=\sum_{\alpha}\res_{z=\alpha}\Phi(z)\,\omega^g_{1}(z)
\]
The dilaton equation still holds for irregular spectral curves whereas the string equations no longer hold. The failure of the string equations can be explicitly observed for the curve $xy^2=1$.

\subsection{Irregular spectral curves}  \label{sec:irr}

One can classify the local behaviour of a spectral curve near a zero of $\dd x$ into four types~---~one of these is regular and the other three are irregular. In all four cases, one can define multidifferentials $\omega^g_{n}$ using equation~\eqref{EOrec}. If $\alpha$ is a zero of $\dd x$, then one of the following four cases must occur.
\begin{enumerate}
\item {\bf Regular.} The form $\dd y$ is analytic and $\dd y(\alpha) \neq 0$. \\ Equivalently, $\alpha$ is a regular zero of $\dd x$ if it is a smooth point of $C$. In this case, there is a pole of $\omega^g_n$ at $\alpha$ of order $6g-4+2n$~\cite{EOrInv}. 
\item {\bf Irregular.}
\begin{enumerate}
\item The form $\dd y$ is analytic at $\alpha$ and $\dd y(\alpha)=0$. \\
This case is ill-behaved since $\omega^g_{n}(z_1, \ldots, z_n)$ loses the key property of symmetry under permutations of $z_1, \ldots, z_n$. (Note that the symmetry of $\omega^g_n$ is not a priori apparent, since the recursion of equation~\eqref{EOrec} treats $z_1$ as special.) For example, if we consider the rational spectral curve given parametrically by $x(z) = z^2$ and $y(z) = z^3$, then topological recursion yields
\[
\omega^0_3(z_1,z_2,z_3) = \frac{1}{2z_1^4 z_2^4 z_3^4} \left[ 3z_1^2 z_2^2 + 3z_1^2 z_3^2 + z_2^2 z_3^2 - 4z_1z_2z_3 \right].
\]
\item The meromorphic function $y$ has a pole at $\alpha$ of order greater than one. \\
In this case, the kernel $K(z_0,z)$ defined above has no pole at $\alpha$ due to the pole of $y$ that appears in the denominator. The residue at $\alpha$ in equation~\eqref{EOrec} therefore vanishes and one obtains no contribution from a neighbourhood of $\alpha$. The invariants in this case match those of the local spectral curve obtained by removing the point $\alpha$.
\item The meromorphic function $y$ has a simple pole at $\alpha$. \\
This case is the main concern of the present paper. Again, the kernel $K(z_0,z)$ defined above has no pole at $\alpha$, but a pole of $\omega_2^0$ at $\alpha$ allows non-zero invariants to survive. The invariants enjoy many of the properties of the invariants for regular curves, such as symmetry of $\omega^g_{n}(z_1, \ldots, z_n)$ under permutations of $z_1, \ldots, z_n$. The pole of $\omega^g_n$ at $\alpha$ is now of order $2g$, which follows from the local analysis in Section~\ref{sec:asym}.
\end{enumerate}
\end{enumerate}

We conclude that the only interesting cases are (1) and (2c), which involve regular zeros of $\dd x$ or a zero of $\dd x$ at which $y$ has a simple pole. If case (2a) is to prove interesting, then one would probably need to adjust the definition of topological recursion in order to recover the symmetry of the invariants.

\section{Enumerating dessins d'enfant}

\subsection{Loop equations}  \label{subsec:loop}

Kazarian and Zograf \cite{KZoVir} prove that $U_g(\mu_1, \ldots, \mu_n)=\mu_1 \cdots \mu_n \, B_{g,n}(\mu_1, \ldots, \mu_n)$ satisfies the recursion
\begin{equation}  \label{rec}
U_g(\mu_1,\mmu_S)=\sum_{j=2}^n\mu_j\,U_g(\mu_1+\mu_j-1,\mmu_{S\setminus\{j\}})
+\sum_{i+j=\mu_1-1} \bigg[ U_{g-1}(i,j,\mmu_S) + \mathop{\sum_{g_1+g_2=g}}_{I \sqcup J = S} U_{g_1}(i,\mmu_I) \, U_{g_2}(j,\mmu_J) \bigg]
\end{equation}
for $S=\{2, \ldots, n\}$ and the base case $U_0(0)=1$. The proof uses an elementary cut-and-join argument that is a variation of the Tutte recursion~\cite{EOrTop}.   Note that $U_0(\mu) = C_\mu = \frac{1}{\mu+1}\binom{2\mu}{\mu}$ is a Catalan number. This is due to the fact that the recursion of equation~\eqref{rec} in the case $(g,n) = (0,1)$ reproduces the Catalan recursion and initial condition
\[
C_m = \sum_{i+j=m-1} C_iC_j \qquad \text{and} \qquad C_0=1.
\]

The recursion \eqref{rec} is equivalent to the fact that the generating functions
\[
W_g(x_1, \ldots, x_n) = \sum_{\mu_1, \ldots, \mu_n=1}^\infty U_g(\mu_1, \ldots, \mu_n)\prod_{i=1}^n x_i^{-\mu_i-1}
\]
satisfy loop equations
\begin{align}  \label{loop}
W_g(x_1,\xx_S) &= W_{g-1}(x_1,x_1,\xx_S)+ \mathop{\sum_{g_1+g_2=g}}_{I \sqcup J = S} W_{g_1}(x_1,\xx_I) \, W_{g_2}(x_1,\xx_J)\\
&+\sum_{j=2}^n \bigg[ \frac{\partial}{\partial x_j}\frac{W_g(x_1,\xx_{S\setminus\{j\}})-W_g(\xx_S)}{x_1-x_j}+\frac{1}{x_1}\frac{\partial}{\partial x_j}W_g(\xx_S)\bigg]+\frac{\delta_{g,0} \, \delta_{n,1}}{x_1}.
\nonumber %%%
\end{align}
The solution of the loop equations for $(g,n)=(0,1)$ defines the spectral curve via the equation
\[
y = W_0(x) = \sum_{\mu=0}^{\infty} U_0(\mu) \, x^{-\mu-1} = \sum_{\mu=0}^{\infty} \frac{1}{\mu+1}\binom{2\mu}{\mu} \, x^{-\mu-1} = \frac{z}{1+z}, \qquad \text{where } x=z+\frac{1}{z}+2.
\]

The proof of the equivalence of \eqref{rec} and \eqref{loop} is standard and relies on the following observations. 
\begin{itemize}
\item The coefficient of $\prod x_i^{-\mu_i-1}$ in $W_g(x_1, \xx_S)$ is $U_g(\mu_1, \mmu_S)$.
\item The coefficient of $\prod x_i^{-\mu_i-1}$ in $W_{g-1}(x_1,x_1,\xx_S)$ is $\displaystyle\sum_{i+j=\mu_1-1} U_{g-1}(i,j,\mmu_S)$.
\item The coefficient of $\prod x_i^{-\mu_i-1}$ in $W_{g_1}(x_1,x_I)W_{g_2}(x_1,x_J)$ is $\displaystyle\sum_{i+j=\mu_1-1} U_{g_1}(i,\mmu_I) \, U_{g_2}(j,\mmu_J)$.
\item The coefficient of $\prod x_i^{-\mu_i-1}$ in
$\frac{\partial}{\partial x_j}\bigg[\frac{W_g(x_1,\xx_{S\setminus\{j\}})-W_g(\xx_S)}{x_1-x_j}+\frac{W_g(\xx_S)}{x_1}\bigg]$ is $\mu_j\,U_g(\mu_1+\mu_j-1,\mmu_{S\setminus\{j\}})$. This observation uses the fact that
\[
\frac{\partial}{\partial x_j}\left(\frac{x_1^{-k}-x_j^{-k}}{x_1-x_j}+\frac{x_j^{-k}}{x_1}\right)=-\frac{\partial}{\partial x_j}\sum_{m=1}^{k-1}x_1^{m-k-1}x_j^{-m}=\sum_{m=1}^{k-1}m\, x_1^{m-k-1}x_j^{-m-1}.
\]
\end{itemize}

\begin{remark}
The loop equations \eqref{loop} are almost a special case of the loop equations appearing in the work of Eynard and Orantin~\cite[Theorem~7.2]{EOrTop}. Using their notation, equation~\eqref{loop} would correspond to $V'(x)=1$ and $P^{(g)}_n(x_1, \ldots, x_n)$ would not be polynomial in $x_1$. Note that such a choice of $V$ and $P^{(g)}_n$ has no meaning there.
\end{remark}

\subsection{Pruned dessins}  \label{sec:prune}

Define $\mathfrak{b}_{g,n}(\mmu)\subseteq\cb_{g,n}(\mmu)$ to be the set of genus $g$ dessins without vertices of valence 1 and with $n$ labelled boundary components of lengths $2\mu_1, \ldots, 2\mu_n$. We refer to such dessins without vertices of valence 1 as \emph{pruned}. The notion of pruned structures has found applications for various other problems~\cite{DNoPru} and will allow us to prove polynomiality for the dessin enumeration here.

\begin{definition} 
For any $\mmu=(\mu_1, \ldots, \mu_n)\in\bz_+^n$, define
\[
b_{g,n}(\mu_1, \ldots, \mu_n)=\sum_{\Gamma\in\mathfrak{b}_{g,n}(\mmu)}\frac{1}{|{\rm Aut\ }\Gamma|}.
\]
\end{definition}

\begin{proposition} \label{prop:pruned}
The numbers $b_{g,n}(\mu_1, \ldots, \mu_n)$ and $B_{g,n}(\mu_1, \ldots,\mu_n)$ are related by the equation
\begin{equation}  \label{pruned}
\sum_{\nu_1, \ldots, \nu_n=1}^\infty b_{g,n}(\nu_1, \ldots, \nu_n)\prod_{i=1}^n z_i^{\nu_i} = \sum_{\mu_1, \ldots, \mu_n=1}^\infty B_{g,n}(\mu_1, \ldots, \mu_n)\prod_{i=1}^n x_i^{-\mu_i},
\end{equation}
where $x_i=z_i+\frac{1}{z_i}+2$. The two sides are analytic expansions of the generating function of equation~\eqref{genfn} at $z_1 = \cdots = z_n = 0$.
\end{proposition}

\begin{proof}
The main idea is that a dessin can be created from a pruned dessin by gluing planar trees to the boundary components. The bicolouring of the vertices extends in a unique way to the additional trees. Conversely, one obtains a unique pruned dessin from a dessin via the process of pruning~---~in other words, repeatedly removing degree one vertices and their incident edges until no more exist.

By gluing planar trees, the count $b_{g,n}(\nu_1, \ldots, \nu_n)$ contributes to $B_{g,n}(\nu_1+k_1, \ldots, \nu_n+k_n)$ for all non-negative integers $k_1, \ldots, k_n$. The contribution is equal to $b_{g,n}(\nu_1, \ldots, \nu_n)$ multiplied by the number of ways to glue planar trees with a total of $k_1$ edges to boundary component 1, multiplied by the number of ways to glue planar trees with a total of $k_2$ edges to boundary component 2, and so on.

The number of ways to glue $k$ edges to a boundary component of length $b$ can be computed as follows. It is simply the number of ways to pick rooted planar trees $T_1, T_2, \ldots, T_b$ with $k$ edges in total. There are $C_i$ rooted planar trees with $i$ edges, where $C_0 = 1, C_1 = 2, C_3 = 5, C_4 = 14, \ldots$ is the sequence of Catalan numbers. So the number of ways to choose rooted planar trees $T_1, T_2, \ldots, T_b$ with $k$ edges in total is simply the $x^b$ coefficient of $(C_0 + C_1X + C_2X^2 + C_3X^3 + \cdots)^b = \left( \frac{1-\sqrt{1-4X}}{2X} \right)^b = f(X)^b$. If we call this number $C_k^b$, we obtain the following formula.\footnote{In fact, one can show that $C_k^b = \frac{b}{b+k}\,\binom{b-1+2k}{k}$.}
\[
B_{g,n}(\mu_1, \ldots, \mu_n) = \sum_{k_1, \ldots, k_n = 0}^\infty b_{g,n}(\mu_1-k_1, \ldots, \mu_n-k_n) \, C^{2\mu_1-2k_1}_{k_1} \cdots C^{2\mu_n-2k_n}_{k_n}
\]

Therefore, we have the following chain of equalities.
\begin{align*}
\sum_{\mu_1, \ldots, \mu_n = 1}^\infty B_{g,n}(\mu_1, \ldots, \mu_n) \prod_{i=1}^n x_i^{-\mu_i} &= \sum_{\mu_1, \ldots, \mu_n = 1}^\infty \sum_{k_1, \ldots, k_n = 0}^\infty b_{g,n}(\mu_1-k_1, \ldots, \mu_n-k_n) \prod_{i=1}^n C^{2\mu_i-2k_i}_{k_i} x_i^{-\mu_i} \\
&= \sum_{\nu_1, \ldots, \nu_n = 1}^\infty \sum_{k_1, \ldots, k_n = 0}^\infty b_{g,n}(\nu_1, \ldots, \nu_n) \prod_{i=1}^n C^{2\nu_i}_{k_i} x_i^{-\nu_i-k_i} \\
&= \sum_{\nu_1, \ldots, \nu_n = 1}^\infty b_{g,n}(\nu_1, \ldots, \nu_n) \prod_{i=1}^n x_i^{-\nu_i} \sum_{k_i= 0}^\infty C^{2\nu_i}_{k_i} x_i^{-k_i} \\
&= \sum_{\nu_1, \ldots, \nu_n = 1}^\infty b_{g,n}(\nu_1, \ldots, \nu_n) \prod_{i=1}^n \left(\frac{f(x_i^{-1})^2}{x_i}\right)^{\nu_i}
\end{align*}

It remains to show that
\[
z_i = \frac{f(x_i^{-1})^2}{x_i} = \frac{x_i - 2 - \sqrt{x_i^2 - 4x_i}}{2},
\]
which follows directly from the relation $x_i = z_i + \frac{1}{z_i} + 2$.
\end{proof}

The next proposition shows that $b_{g,n}$ satisfies a \emph{functional recursion}, in the sense that it involves only terms with simpler $(g,n)$ complexity on the right hand side. This will be useful in understanding the pole structure of the generating function~\eqref{genfn}. Such a recursion is in contrast with equation~\eqref{rec} in the non-pruned case, which includes $(g,n)$ terms on both sides of the equation.

%Define a pruned version of the dessin count $n_{k,\ell}(\mu)$ to exclude dessins with valence 1 vertices and define $\displaystyle n_{g,n}(\mu)=\hspace{-6mm}\sum_{k+\ell=d+2-2g-n}\hspace{-4mm}n_{k,\ell}(\mu)$ where $\mu=(\mu_1, \ldots, \mu_n)$.   

\begin{proposition}  \label{th:recprune}
The pruned dessin enumeration satisfies the following recursion for $(g,n) \neq (0,1), (0,2), (0,3), (1,1)$.
\begin{align*}
|\mmu| \, b_{g,n}(\mmu) &= \sum_{i=1}^n \sum_{p+q+r=\mu_i} pqr \bigg[ b_{g-1,n+1}(p, q, \mmu_{S \setminus \{i\}}) + \mathop{\sum_{g_1+g_2=g}}^{\mathrm{stable}}_{I \sqcup J = S \setminus \{i\}} b_{g_1, |I|+1}(p, \mmu_I) \, b_{g_2, |J|+1}(q, \mmu_J) \bigg] \\
&+ \sum_{i \neq j} \sum_{p+q=\mu_i+\mu_j} pq \, b_{g,n-1}(\mmu_{S \setminus \{i,j\}})
\end{align*}
Here, $S = \{1, 2, \ldots, n\}$ and we set $\mmu_I = (\mu_{i_1}, \mu_{i_2}, \ldots, \mu_{i_k})$ for $I = \{i_1, i_2, \ldots, i_k\}$. The word \emph{stable} over the summation indicates that we exclude all terms that involve $b_{0,1}$ or $b_{0,2}$.
\end{proposition}

\begin{proof}
We count dessins in ${\mathfrak b}_{g,n}(\mmu)$ with a marked edge. The most obvious way to count such objects is to choose a dessin in the set, which can be accomplished in $b_{g,n}(\mmu)$ ways, and then to choose a suitable edge, which can be accomplished in $|\mmu|$ ways. So the total number of such objects is $|\mmu| \, b_{g,n}(\mmu)$, which forms the left hand side of the recursion.

To form the right hand side of the recursion, we count the same objects in the following way. For a dessin in ${\mathfrak b}_{g,n}(\mmu)$ with a marked edge, remove the marked edge and repeatedly remove degree 1 vertices and their incident edges to obtain a pruned dessin. One of the following three cases must arise.

\begin{itemize}
\item {\bf Case 1.} The marked edge is adjacent to face $i$ on both sides and its removal leaves a connected dessin. \\
Suppose that $r$ edges are removed in total~---~they necessarily form a path. The resulting pruned dessin must lie in the set ${\mathfrak b}_{g-1, n+1}(p, q, \mmu_{S \setminus \{i\}}))$, where $p + q + r = \mu_i$.

\noindent Conversely, there are $pqr$ ways to reconstruct a marked pruned dessin in ${\mathfrak b}_{g,n}(\mmu)$ from a pruned dessin in the set ${\mathfrak b}_{g-1, n+1}(p, q, \mmu_{S \setminus \{i\}}))$, where $p + q + r = \mu_i$, by adding a path of $r$ edges. The factor $r$ accounts for choosing a marked edge along the path. The factors $p$ and $q$ account for the choice of endpoints of the path.

\item {\bf Case 2.} The marked edge is adjacent to face $i$ on both sides and its removal leaves the disjoint union of two dessins. \\
Suppose that $r$ edges are removed in total~---~they necessarily form a path. Suppose that the two components have faces $I$ and $J$. The resulting pruned dessins must lie in the sets ${\mathfrak b}_{g_1, |I|+1}(p, \mmu_I))$ and ${\mathfrak b}_{g_2, |J|+1}(q, \mmu_J))$, where $p + q + r = \mu_i$. Furthermore, we cannot obtain a pruned dessin of type $(0,1)$ in this manner.

\noindent Conversely, there are $pqr$ ways to reconstruct a marked pruned dessin in ${\mathfrak b}_{g,n}(\mmu)$ from two pruned dessins in the sets ${\mathfrak b}_{g_1, |I|+1}(p, \mmu_I))$ and ${\mathfrak b}_{g_2, |J|+1}(q, \mmu_J))$, where $p + q + r = \mu_i$, by adding a path of $r$ edges. The factor $r$ accounts for choosing a marked edge along the path. The factors $p$ and $q$ account for the choice of endpoints of the path.

\item {\bf Case 3.} The marked edge is adjacent to faces $i$ and $j$, for $i \neq j$. \\
Suppose when walking along the marked edge from the white vertex to the black vertex that face $i$ lies on the left and face $j$ on the right. Suppose that $q$ edges are removed in total~---~they necessarily form a path. The resulting pruned dessin must lie in the set ${\mathfrak b}_{g, n-1}(p, \mmu_{S \setminus \{i,j\}}))$, where $p + q = \mu_i + \mu_j$.

\noindent Conversely, there are $pq$ ways to reconstruct a marked pruned dessin in ${\mathfrak b}_{g,n}(\mmu)$ from a pruned dessin in the set ${\mathfrak b}_{g, n-1}(p, \mmu_{S \setminus \{i,j\}}))$, where $p + q = \mu_i + \mu_j$, by adding a path of $q$ edges. The factor $q$ accounts for choosing an edge along the path. The factor $p$ arises from choosing where to glue one of the ends of the path. Note that there is a unique choice to glue in the other end to create a face of perimeter $\mu_i$ on the left and a face of perimeter $\mu_j$ on the right.
\end{itemize}

There is a crucial subtlety that arises in the third case, which we now address. One can discern the issue by considering the sequence of diagrams below, in which $\mu_i$ increases from left to right, relative to $\mu_j$.
\begin{center}
\includegraphics{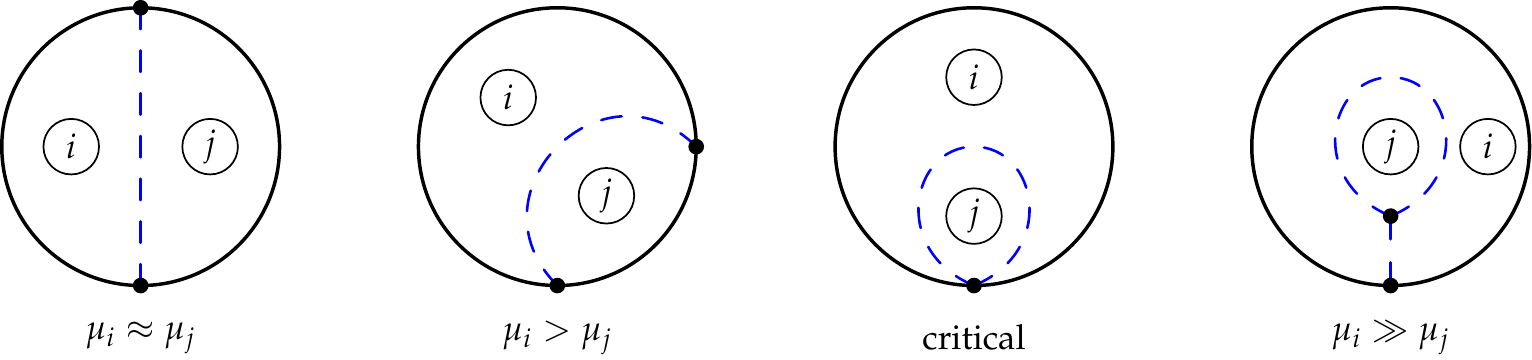}
\end{center}
The third case actually contributes to diagrams like the one on the far right, in which face $i$ completely surrounds face $j$, or vice versa. In fact, the edge that we remove can lie anywhere along the dashed path in the schematic diagram. Note that this contributes to the second case, in which the marked edge is adjacent to the face labelled $i$ on both sides and its removal leaves the disjoint union of two connected graphs. However, observe that this surplus contribution is precisely equal to the terms from the second case that involve $b_{0,2}$, so one can compensate simply by excluding such terms. Given that we have already witnessed that $b_{0,1} = 0$, we can restrict to the so-called {\em stable} terms in the second case, which are precisely those that do not involve $b_{0,1}$ or $b_{0,2}$.

Therefore, to obtain all marked dessins in ${\mathfrak b}_{g,n}(\mmu)$ exactly once, it is necessary to perform the reconstruction process
\begin{itemize}
\item in the first case for all values of $i$ and $p + q + r = \mu_i$;
\item in the second case for all {\em stable} values of $i$, $p + q + r = \mu_i$, $g_1 + g_2 = g$, and $I \sqcup J = S \setminus \{i\}$; and
\item in the third case for all values of $i$, $j$, and $p + q = \mu_i + \mu_j$.
\end{itemize}
We obtain the desired recursion by summing up over all these contributions.
\end{proof}

\begin{example}  
Calculation of $b_{1,1}(\mu_1)$ builds dessins from loops of circumference $p$.   
\[
2\mu_1\,b_{1,1}(\mu_1)=\frac{1}{2} \mathop{\sum_{2p+q=2 \mu_1}}_{p \text{ even}} pq
=\left\{\begin{array}{ll}\frac{1}{12}\mu_1(\mu_1^2-4)&\mu_1{\rm\ even}\\\frac{1}{12}\mu_1(\mu_1^2-1)&\mu_1{\rm\ odd}\end{array}\right.
\]
Note that this is precisely the recursion of Proposition~\ref{th:recprune}, using the fact that $b_{0,2}(\mu_1, \mu_2) = \frac{\delta(\mu_1, \mu_2)}{\mu_1}$. Taking this definition, one can also apply the recursion in the case $(g,n) = (0,3)$. Finally we obtain
\begin{align*}
b_{0,3}(\mu_1, \mu_2, \mu_3) &= 2, \\
b_{1,1}(\mu_1) &= \left\{\begin{array}{ll}\frac{1}{24}(\mu_1^2-4),&\mu_1{\rm\ even} \\ 
\frac{1}{24}(\mu_1^2-1),&\mu_1{\rm\ odd}.\end{array}\right.
\end{align*}
\end{example}

Calculation of $b_{g,n}$ for small values of $g$ and $n$ indicates that it is a quasi-polynomial modulo 2, although this structure does not follow immediately from the recursion above. In order to prove it, we use the following asymmetric version of the recursion.

\begin{proposition} \label{prop:precursion}
The pruned dessin enumeration satisfies the following recursion for $(g,n) \neq (0,1), (0,2), (0,3), (1,1)$.
\begin{align*}
\mu_1 b_{g,n}(\mu_1, \mmu_S) &= \sum_{p+q+r=\mu_1} pqr \bigg[ b_{g-1,n+1}(p, q, \mmu_S) + \mathop{\sum_{g_1+g_2=g}}^{\mathrm{stable}}_{I \sqcup J = S} b_{g_1, |I|+1}(p, \mmu_I) \,b_{g_2, |J|+1}(q, \mmu_J) \bigg] \\
&+ \sum_{i \in S} \bigg[ \sum_{p+q=\mu_1+\mu_i} pq \left.b_{g,n-1}(\mmu_S)\right|_{\mu_i=p} + \mathrm{sign}(\mu_1-\mu_i) \sum_{p+q=|\mu_1-\mu_i|} pq \left.b_{g,n-1}(\mmu_S)\right|_{\mu_i=p} \bigg]
\end{align*}
Here, $S = \{2, \ldots, n\}$ and for $I = \{i_1, i_2, \ldots, i_k\}$, we set $\mmu_I = (\mu_{i_1}, \mu_{i_2}, \ldots, \mu_{i_k})$.
\end{proposition}

\begin{proof}
Use the fact that both the symmetric recursion given by Proposition~\ref{th:recprune} and the asymmetric recursion here uniquely determine all $b_{g,n}(\mmu)$ from the base cases $b_{0,3}$ and $b_{1,1}$. So it suffices to show that the symmetric version follows from the asymmetric version, and this can be seen by symmetrising.
\end{proof}

\begin{corollary}
For $(g,n) \neq (0,1)$ or $(0,2)$, $b_{g,n}(\mu_1, \ldots, \mu_n)$ is a quasi-polynomial modulo 2 of degree $3g-3+n$ in $\mu_1^2, \ldots, \mu_n^2$.
\end{corollary}

\begin{proof}
To show that $b_{g,n}$ is a polynomial in the squares, we use the following two facts, which are straightforward to verify.
\begin{itemize}
\item For all non-negative integers $a$ and $b$, the functions
\[
f_{a,b}^{(0)}(\mu) = \mathop{\sum_{p+q+r=\mu}}_{p + q \text{ even}} p^{2a+1}q^{2b+1}r \qquad \text{and} \qquad f_{a,b}^{(1)}(\mu) = \mathop{\sum_{p+q+r=\mu}}_{p + q \text{ odd}} p^{2a+1}q^{2b+1}r
\]
are odd quasi-polynomials modulo 2 in $\mu$ of degree $2a+2b+5$.
\item For all non-negative integers $a$, the functions
\begin{align*}
g_a^{(0)}(\mu_1, \mu_2) &= \mathop{\sum_{p+q=\mu_1+\mu_2}}_{p \text{ even}} p^{2a+1}q + \mathrm{sign}(\mu_1-\mu_2) \mathop{\sum_{p+q=|\mu_1-\mu_2|}}_{p \text{ even}} p^{2a+1}q \\
g_a^{(1)}(\mu_1, \mu_2) &= \mathop{\sum_{p+q=\mu_1+\mu_2}}_{p \text{ odd}} p^{2a+1}q  + \mathrm{sign}(\mu_1-\mu_2) \mathop{\sum_{p+q=|\mu_1-\mu_2|}}_{p \text{ odd}} p^{2a+1}q
\end{align*}
are quasi-polynomials modulo 2 that are odd in $\mu_1$ and even in $\mu_2$ of degree $2a+5$.
\end{itemize}

%For example, with the notation $\partial = x \frac{\partial}{\partial x}$,
%\begin{align*}
%f_{a,b}^{(0)}(\mu) &= \mathop{\sum_{p+q+r=\mu}}_{p + q \text{ even}} p^{2a+1}q^{2b+1}r \\
%&= [x^\mu] \left[ \sum_p p^{2a+1} x^p \sum_q q^{2b+1} x^q \right]_{\text{even}} \left[ \sum_r rx^r \right] \\
%&= [x^\mu] \left[ \partial^{2a+1} \frac{x}{1-x} \times \partial^{2b+1} \frac{x}{1-x} \right]_{\text{even}} \left[ \partial \frac{x}{1-x} \right] \\
%\end{align*}

By the previous example, the base cases $b_{0,3}(\mu_1, \mu_2, \mu_3)$ and $b_{1,1}(\mu_1)$ are indeed even quasi-polynomials modulo 2 of degrees 0 and 1, respectively. Now consider $b_{g,n}$ satisfying $2g-2+n \geq 2$ and suppose that the proposition is true for all $b_{g,n}$ of lesser complexity. Then the recursion of Proposition~\ref{prop:precursion} expresses $\mu_1 b_{g,n}(\mu_1, \ldots, \mu_n)$ as a finite linear combination of terms of the form
\[
f^{(c)}_{a,b}(\mu_1) \prod_{k \in S} \mu_k^{2a_k} \qquad \text{and} \qquad g_a^{(c)}(\mu_1, \mu_i) \prod_{k \in S \setminus \{i\}} \mu_k^{2a_k},
\]
where $a$ and $b$ are non-negative integers and $c \in \{0, 1\}$. Upon dividing by $\mu_1$, we find that $b_{g,n}(\mu_1, \ldots, \mu_n)$ is a quasi-polynomial modulo 2 that is even in $\mu_1, \ldots, \mu_n$. Furthermore, one can check that it is of the correct degree. Therefore, we have proven the proposition by induction on $2g-2+n$.
\end{proof}

\begin{center}
\begin{tabular}{cccl} \toprule
$g$ & $n$ & condition on $(\mu_1, \ldots, \mu_n)$ & $b_{g,n}(\mu_1, \ldots, \mu_n)$ \\ \midrule
0 & 1 & none & 0 \\
0 & 2 & none & $\frac{\delta(\mu_1, \mu_2)}{\mu_1}$ \\
0 & 3 & none & 2 \\
0 & 4 & none & $\mu_1^2 + \mu_2^2 + \mu_3^2 + \mu_4^2 - 1$ \\
1 & 1 & $\mu_1$ even & $\frac{1}{24}(\mu_1^2-4)$ \\
1 & 1 & $\mu_1$ odd & $\frac{1}{24}(\mu_1^2-1)$ \\
1 & 2 & $\mu_1 + \mu_2$ even & $\frac{1}{48} (\mu_1^2 + \mu_2^2 - 2) (\mu_1^2 + \mu_2^2 - 4)$ \\
1 & 2 & $\mu_1 + \mu_2$ odd & $\frac{1}{48} (\mu_1^2 + \mu_2^2 - 1) (\mu_1^2 + \mu_2^2 - 5)$ \\
2 & 1 & $\mu_1$ even & $\frac{1}{276480} (\mu_1^2-4) (\mu_1^2-16) (5\mu_1^4 - 38\mu_1^2 + 72)$ \\
2 & 1 & $\mu_1$ odd & $\frac{1}{276480} (\mu_1^2-1) (\mu_1^2-9) (5\mu_1^4 - 88\mu_1^2 + 227)$ \\ \bottomrule
\end{tabular}
\end{center}

%\begin{remark} We may be able to prove polynomial structure directly from the symmetric recursion. The symmetric recursion implies that $b_{g,n}(\mu_1, \ldots, \mu_n)$ is a rational function of $\mu_1, \ldots, \mu_n$. Note that an automorphism of a dessin gives an automorphism of the underlying fatgraph, whose vertices all have degree at least three. However, the number of such fatgraphs of type $(g,n)$ is finite. It follows that the size of the automorphism group of a dessin of type $(g,n)$ is bounded. Therefore, $m(g,n) \, b_{g,n}(\mu_1, \ldots, \mu_n) \in \mathbb{Z}$ for some integer $m(g,n)$. However, any such rational function must be a polynomial. (This argument needs to be checked and fixed to take into account the quasi-polynomiality.) However, this doesn't seem to give us polynomiality in the squares. \end{remark}

%%%%%%%%%%%%%%%%%%%%%%%%%%%%%%%%%%%%%%%%%%%%

The structure theorem for $b_{g,n}(\mu_1, \ldots, \mu_n)$ implies the following structure theorem for 
\[
\Omega_{g,n}(z_1, \ldots, z_n) = \frac{\partial}{\partial x_1} \cdots \frac{\partial}{\partial x_n} F_{g,n}(x_1, \ldots, x_n)~\dd x_1 \otimes \cdots \otimes \dd x_n
\]
via Proposition~\ref{prop:pruned}. This will play an important role in the proof of Theorem~\ref{th:main}.

\begin{proposition}
For $(g,n) \neq (0,1)$ or $(0,2)$, $\Omega_{g,n}(z_1, \ldots, z_n)$ is a meromorphic multidifferential on the rational curve $x=z+\frac{1}{z}+2$, with poles only at $z_i=\pm 1$. Furthermore, it satisfies the skew invariance property
\[
\Omega_{g,n}(z_1, \ldots, \tfrac{1}{z_i}, \ldots, z_n)=-\Omega_{g,n}(z_1, \ldots, z_n), \qquad \text{for } i = 1, 2, \ldots, n.
\]
\end{proposition}

\begin{proof}
We begin with the result of Proposition~\ref{pruned}.
\[
F_{g,n}(x_1, \ldots, x_n) = \sum_{\nu_1, \ldots, \nu_n=1}^\infty b_{g,n}(\nu_1, \ldots, \nu_n)\prod_{i=1}^n z_i^{\nu_i}
\]
The structure theorem for $b_{g,n}$ allows us to express this as a linear combination of terms of the form
\[
\prod_{i=1}^n f_{k_i}^{(s_i)}(z_i),
\]
where
\[
f_k^{(0)}(z) = \sum_{\nu \text{ even}} \nu^{2k} z^\nu = \left( z \frac{\partial}{\partial z} \right)^{2k} \frac{z^2}{1-z^2} \qquad \text{and} \qquad f_k^{(1)}(z) = \sum_{\nu \text{ odd}} \nu^{2k} z^\nu = \left( z \frac{\partial}{\partial z} \right)^{2k} \frac{z}{1-z^2}.
\]

Note that
\[
f_0^{(0)}(z) + f_0^{(0)}(\tfrac{1}{z}) = \frac{z^2}{1-z^2} + \frac{1/z^2}{1-1/z^2} = -1 \qquad \text{and} \qquad f_0^{(1)}(z) + f_0^{(0)}(\tfrac{1}{z}) = \frac{z}{1-z^2} + \frac{1/z}{1-1/z^2} = 0.
\]
Since $\left(z \frac{\partial}{\partial z} \right)^2 = \left(w \frac{\partial}{\partial w} \right)^2$ for $w = \frac{1}{z}$, we have
\[
f_k^{(s)}(z) + f_k^{(s)}(\tfrac{1}{z}) = 0,
\]
for $s = 0, 1$ and $k \geq 1$. It follows that 
\[
F_{g,n}(x_1, \ldots, x_n) = - \left. F_{g,n}(x_1, \ldots, x_n) \right|_{z_i \mapsto \tfrac{1}{z_i}} + [\,\text{terms independent of $z_i$}\,].
\]
Now apply the total derivative to both sides to obtain the skew invariance property for $\Omega_{g,n}$.
\end{proof}

\begin{remark}%\subsection{Relation to lattice point count.}
The pruned version of the enumeration of fatgraphs $M_{g,n}(\mmu)$ defined in the introduction gives rise to
\[
N_{g,n}(\mu_1, \ldots, \mu_n)=\sum_{\Gamma\in{\mathfrak f}_{g,n}(\mmu)}\frac{1}{|{\rm Aut\ }\Gamma|}.
\]
for ${\mathfrak f}_{g,n}(\mmu) \subseteq \fat(\mmu)$ defined as the subset of connected genus $g$ fatgraphs without vertices of valence 1. This was studied in \cite{NorCou,NorStr} and shown to satisfy
\begin{itemize}
\item
$N_{g,n}(\mu_1, \ldots, \mu_n)$ is a degree $6g-6+2n$ quasi-polynomial modulo 2%---polynomial on each coset of $2\bz^n\subset\bz^n$
\item $N_{g,n}(\mu_1, \ldots, \mu_n)=0$ for $|\mmu|$ odd;
\item $N_{g,n}(0, \ldots, 0)=\chi(\modm_{g,n})$; and
\item highest coefficients of $N_{g,n}$ are psi-class intersection numbers on $\overline{\modm}_{g,n}$.
%\item recursively calculable $\Rightarrow$ proof of Witten-Kontsevich theorem
%\item counts curves defined over $\bar{\bq}$
%\item $\displaystyle \omega^{(g)}_n({\bf z})=\prod d_{z_i}\sum_{{\bf b}>0} N_{g,n}({\bf b})z_1^{\mu_1} \cdots z_n^{\mu_n}$ satisfies TR---spectral curve $x=z+1/z$, $y=z$.  
\end{itemize}

Genus 0 fatgraphs with even length boundary components admit bipartite colourings, and there are exactly two ways to bicolour the vertices, so
\[
b_{0,n}(\mmu)=2N_{0,n}(2\mmu).
\]
The $N_{g,n}$ satisfy a recursion similar to that in Proposition~\ref{th:recprune} and the two recursions coincide in genus 0 with all $\mu_i$ even.  The recursion relation does not specialise to the case of even $\mu_i$ in general --- for example, $N_{1,1}(\mu_1)$ requires $N_{0,3}(\mu_1,\mu_2,\mu_3)$ with some $\mu_i$ odd.  The non-bipartite enumeration mentioned in the introduction can be obtained as
\begin{equation}  \label{nonbip}
\nb_g(\mu_1, \ldots, \mu_n)=2N_{g,n}(2\mu_1, \ldots, 2\mu_n)-b_{g,n}(\mu_1, \ldots, \mu_n),
\end{equation}
which vanishes when $g=0$. It is worth noting that $N_{g,n}(2\mu_1, \ldots, 2\mu_n)$ is a polynomial in $\mu_1, \ldots, \mu_n$ rather than a quasi-polynomial.  
\end{remark}

\begin{remark}
For positive genus, we have $b_{g,n}(\mu)\neq2N_{g,n}(2\mu)$ in general. To check that a given fatgraph $\Gamma$ is bipartite, one needs to check that $\beta_1(\Gamma)$ independent cycles have even length, where
\[
\beta_1(\Gamma) = e(\Gamma) - v(\Gamma) + 1 = 2g(\Gamma) - 1 + n(\Gamma)
\]
is the first Betti number of $\Gamma$. Here, $e, v, g, n$ denote the number of edges, number of vertices, genus, and number of boundary components, respectively.

For $g = 0$, one can perform this check on $n(\Gamma) - 1$ cycles formed by boundary components. Hence, we obtain the relation $b_{0,n}(\mu)=2N_{0,n}(2\mu)$, as desired. For general genus, we can check on $n(\Gamma) - 1$ face cycles formed bay boundary components. However, it is still necessary to check on $2g$ more independent cycles, and each of these imposes an independent parity condition. So we obtain the asymptotic relation
\[
b_{g,n}(\mmu) \sim \frac{1}{2^{2g}} 2N_{g,n}(2\mmu).
\]
We know that $N_{g,n}(2\mmu)$ is polynomial in $\mu_1, \ldots, \mu_n$ with leading coefficients for $a_1 + \cdots + a_n = 3g-3+n$ given by
\[
[\mu_1^{2a_1} \cdots \mu_n^{2a_n}] N_{g,n}(2\mmu) = \frac{2^g}{a_1! \cdots a_n!} \int_{\overline{\mathcal M}_{g,n}} c_1(\cl_1)^{a_1} \cdots c_1(\cl_n)^{a_n}.
\]
It follows that
\[
[\mu_1^{2a_1} \cdots \mu_n^{2a_n}] b_{g,n}(\mmu) = \frac{1}{2^{g-1}} \frac{1}{a_1! \cdots a_n!} \, \int_{\overline{\mathcal M}_{g,n}} c_1(\cl_1)^{a_1} \cdots c_1(\cl_n)^{a_n}.
\]
\end{remark}

%\begin{remark} If a spectral curve $(C,B,x,y)$ admits an action of a finite group $G$, for example $(x,y)=(z+1/z,z)$ with standard $B=dzdz'/(z-z')^2$ is symmetric with respect to the $\bz_2$-action $z\mapsto -z$  Is \eqref{spec} interesting because it matches the invariant part of $x=z+1/z$, $y=z$? Note: $B$ is {\em not} invariant under $(z,z')\mapsto (-z,z')$. \end{remark}

\subsection{Proof of topological recursion}

\begin{proof}[Proof of Theorem~\ref{th:main}]
The proof that the loop equations \eqref{loop} imply topological recursion for the spectral curve $x=z+\frac{1}{z}+2$ and $y=\frac{z}{1+z}$ is standard. If we write $y=W_0(x)$, then \eqref{loop} implies for $2g-2+n>0$ that
\begin{align*}  
W_g(x_1,\xx_S) - 2y\,W_g(x_1,\xx_S)&=W_{g-1}(x_1,x_1,\xx_S) + \mathop{\sum^{\text{stable}}_{g_1+g_2=g}}_{I \sqcup J = S} W_{g_1}(x_1,x_I) \, W_{g_2}(x_1,x_J)\\
+\sum_{j=2}^n \bigg[2\,W_0&(x_1,x_j) \, W_g(x_1,\xx_{S\setminus\{j\}}) + \frac{\partial}{\partial x_j}\frac{W_g(x_1,\xx_{S\setminus\{j\}})}{x_1-x_j}-\frac{\partial}{\partial x_j}\frac{W_g(\xx_S)}{x_1-x_j}+\frac{1}{x_1}\frac{\partial}{\partial x_j}W_g(\xx_S)\bigg]\\
\Rightarrow \frac{1-z}{1+z} \,W_g(x_1,\xx_S) \, \dd x_1^{\otimes 2} \, \dd \xx_S &= W_{g-1}(x_1,x_1,\xx_S) \, \dd x_1^{\otimes 2} \, \dd \xx_S + \mathop{\sum^{\text{stable}}_{g_1+g_2=g}}_{I \sqcup J = S} W_{g_1}(x_1,x_I) \, W_{g_2}(x_1,x_J) \, \dd x_1^{\otimes 2}\, \dd \xx_S\\
+\sum_{j=2}^n \bigg[ 2\,W_0(x_1,x_j) \, W_g&(x_1,\xx_{S\setminus\{j\}}) + \frac{\partial}{\partial x_j}\frac{W_g(x_1, \xx_{S\setminus\{j\}})}{x_1-x_j}-\frac{\partial}{\partial x_j}\frac{W_g(\xx_S)}{x_1-x_j}+\frac{1}{x_1}\frac{\partial}{\partial x_j}W_g(\xx_S)\bigg] \dd x_1^{\otimes 2} \, \dd \xx_S
\end{align*}
The word \emph{stable} over the summation indicates that we exclude all terms that involve $W_0(x_1)$ or $W_0(x_1, x_i)$.

From Section~\ref{sec:prune}, we know that $\Omega_g(z_1, \ldots, z_n)=W_g(x_1, \ldots, x_n)\,\dd x_1 \cdots \dd x_n$ is a meromorphic multidifferential on the rational curve $x=z+\frac{1}{z}+2$, with poles only at $z_i=\pm 1$ and satisfying skew invariance in the sense that  % do we know the skew invariance from section 3.2?
\[
\Omega_g(z_1, \ldots, \tfrac{1}{z_i}, \ldots, z_n)=-\Omega_g(z_1, \ldots, z_i, \ldots, z_n), \qquad \text{for } i = 1, 2, \ldots, n.
\]

Put $\omega^0_2(z,w)=\frac{\dd z \, \dd w}{(z-w)^2}$ and note that
\[
\bigg[2\,W_0(x_1,x_j)+\frac{1}{(x_1-x_j)^2}\bigg]\dd x_1 \, \dd x_j = \omega^0_2(z_1,z_j)-\omega^0_2(\tfrac{1}{z_1},z_j).
\]
Hence,
\begin{align*}  
 \Omega_g(z_1,\zz_S)\,\frac{1-z_1}{1+z_1}\,\dd x_1&=\Omega_{g-1}(z_1,z_1,\zz_S)+\mathop{\sum^{\text{stable}}_{g_1+g_2=g}}_{I \sqcup J = S} \Omega_{g_1}(z_1,\zz_I)\,\Omega_{g_2}(z_1,\zz_J)\\
&+\sum_{j=2}^n \bigg[\omega^0_2(z_1,z_j)-\omega^0_2(\tfrac{1}{z_1},z_j)\bigg]\Omega_g(z_1,\zz_{S\setminus\{j\}})
-\frac{\partial}{\partial x_j}\sum_{j=2}^n\frac{W_g(\xx_S)\,x_j}{(x_1-x_j)\,x_1} \, \dd x_1^{\otimes 2}\,\dd \xx_S\\
&=-\Omega_{g-1}(z_1,\tfrac{1}{z_1},\zz_S) - \mathop{\sum^{\text{stable}}_{g_1+g_2=g}}_{I \sqcup J = S} \Omega_{g_1}(z_1,\zz_I) \, \Omega_{g_2}(\tfrac{1}{z_1},\zz_J) - \frac{\partial}{\partial x_j}\sum_{j=2}^n\frac{W_g(\xx_S)\,x_j}{(x_1-x_j)\,x_1}\,\dd x_1^{\otimes 2}\,\dd \xx_S \\
&- \sum_{j=2}^n \omega^0_2(z_1,z_j) \, \Omega_g(\tfrac{1}{z_1},\zz_{S\setminus\{j\}}) - \omega^0_2(\tfrac{1}{z_1},z_j) \, \Omega_g(z_1,\zz_{S\setminus\{j\}}).
\end{align*}
A rational differential is a sum of its principal parts. Recall that the {\em principal part} of a meromorphic differential $h(z)$ with respect to the rational parameter $z$ at $\alpha\in C$ is 
\[
[h(z)]_{\alpha}:=\res_{w=\alpha}\frac{h(w)\,\dd w}{z-w}=\text{negative part of the Laurent series of } h(z) \text{ at } \alpha.
\]
Hence, express $\Omega_g(z_1, \ldots, z_n)$ as  a sum of its principal parts thus.
\begin{align*}  
\Omega_g(z_1,\zz_S)&=-\sum_{\alpha=\pm1}\res_{z=\alpha}\frac{\dd z}{z_1-z}\frac{1+z}{1-z}\frac{1}{\dd x(z)} \bigg[\Omega_{g-1}(z,\tfrac{1}{z},\zz_S) + \mathop{\sum^{\text{stable}}_{g_1+g_2=g}}_{I \sqcup J = S} \Omega_{g_1}(z,\zz_I) \, \Omega_{g_2}(\tfrac{1}{z},\zz_J)\\
&+\sum_{j=2}^n\omega^0_2(z,z_j) \, \Omega_g(\tfrac{1}{z},\zz_{S\setminus\{j\}}) + \omega^0_2(\tfrac{1}{z},z_j) \, \Omega_g(z, \zz_{S\setminus\{j\}})\bigg]
\end{align*} 
The third term is annihilated by the residue, since
\[
\frac{1+z_1}{1-z_1}
\frac{\partial}{\partial x_j}\sum_{j=2}^n\frac{W_g(\xx_S)\,x_j}{(x_1-x_j)\,x_1} \, \dd x_1 \, \dd \xx_S
\] 
is analytic at $z_1=\pm 1$. This follows from the fact that $\dfrac{1+z}{1-z}\dfrac{\dd x}{x}=-\dfrac{\dd z}{z}$ is analytic at $z=\pm 1$.

Skew invariance under $z\mapsto\frac{1}{z}$ of all terms allows us to express this as
\begin{align*}  
\Omega_g(z_1,\zz_S) &= \sum_{\alpha=\pm1} \res_{z=\alpha} \frac{1}{2} \bigg(\frac{\dd z}{z_1-\frac{1}{z}}-\frac{\dd z}{z_1-z}\bigg)\frac{1+z}{1-z}\frac{1}{\dd x(z)} \bigg[\Omega_{g-1}(z,\tfrac{1}{z},\zz_S)+ \mathop{\sum^{\text{stable}}_{g_1+g_2=g}}_{I \sqcup J = S} \Omega_{g_1}(z,\zz_I) \, \Omega_{g_2}(\tfrac{1}{z},\zz_J)\\
&+\sum_{j=2}^n\omega^0_2(z,z_j) \, \Omega_g(\tfrac{1}{z},z_{S\setminus\{j\}})+\omega^0_2(z^{-1},z_j) \, \Omega_g(z,z_{S\setminus\{j\}})\bigg] \\
&=\sum_{\alpha=\pm1}\res_{z=\alpha}K(z_1,z) \bigg[\Omega_{g-1}(z,\tfrac{1}{z},\zz_S) + \mathop{\sum^\circ_{g_1+g_2=g}}_{I \sqcup J = S} \Omega_{g_1}(z,\zz_I) \, \Omega_{g_2}(\tfrac{1}{z},\zz_J)\bigg].
\end{align*}  
Here, $K(z_1,z) = \frac{1}{2}\left(\frac{\dd z}{z_1-\frac{1}{z}}-\frac{\dd z}{z_1-z}\right)\frac{1+z}{1-z}\frac{1}{\dd x(z)}$ and the $\circ$ over the inner summation means that we exclude terms that involve $\Omega^0_1$. This completes the proof that $\Omega_g(z_1, \ldots, z_n) =\omega^g_n(z_1, \ldots, z_n)$ for the spectral curve $C$ given by equation~\eqref{specurve}.
\end{proof}

\subsection{Polynomial behaviour of invariants}

It is possible to solve the recursion \eqref{rec} explicitly in low genus.

\begin{proposition}  \label{th:formg=0}
In genus 0, we have the explicit formula
\begin{equation}  \label{formg=0}
B_{0,n}(\mu_1, \ldots, \mu_n)=2^{1-n}(|\mmu|-1)(|\mmu|-2) \cdots (|\mmu|-n+3)\prod_{i=1}^n\binom{2\mu_i}{\mu_i},\qquad \text{for } n\geq 3.
\end{equation}
\end{proposition}

It is unclear how to prove Proposition~\ref{th:formg=0} by substituting directly into the recursion.  Instead, we first prove a general structure theorem for $B_{g,n}(\mu_1, \ldots, \mu_n)$ --- namely, that it is a polynomial multiplied by an explicit separable part.

\begin{theorem}  \label{th:struc}
For $(g,n) \neq (0,1)$ or $(0,2)$,
\[
 B_{g,n}(\mu_1, \ldots, \mu_n)=p_{g,n}(\mu_1, \ldots, \mu_n) \prod_{i=1}^nc_g(\mu_i),
 \]
where $p_{g,n}$ is a polynomial of degree $3g-3+n+ng$ and
\[
c_g(\mu)=\frac{(2\mu-2g)!}{\mu!\,(\mu-g)!}=\binom{2\mu}{\mu}2^{-g}\prod_{k=1}^g\frac{1}{2\mu-2k+1}.
\]
The right hand expression for $c_g(\mu)$ allows for evaluation at $\mu=0, 1, \ldots, g-1$. 
\end{theorem}

\begin{remark}
The two unstable cases $(g,n)=(0,1)$ and $(0,2)$ do in fact satisfy a similar structure result if one allows polynomials of negative degree.
\[
 B_{0,1}(\mu_1)=\frac{1}{\mu_1(\mu_1+1)}\binom{2\mu_1}{\mu_1} \qquad \qquad B_{0,2}(\mu_1,\mu_2)=\frac{1}{2(\mu_1+\mu_2)}\binom{2\mu_1}{\mu_1}\binom{2\mu_2}{\mu_2}
\]
\end{remark}

\begin{proof}
The proof simply uses the structure of meromorphic functions on the spectral curve $xy^2-xy+1=0$.  We begin with the genus 0 invariants. From Theorem~\ref{th:main}, it is easy to see by induction that the generating function 
\[
W_0(x_1, \ldots, x_n)=\sum_{\mu_1, \ldots, \mu_n=1}^\infty U_0(\mu_1, \ldots, \mu_n)\prod_{i=1}^n x_i^{-\mu_i-1}
\]
is an expansion of a function that is rational in $z_i$ for $x_i=z_i+\frac{1}{z_i}+2$, with poles only at $z_i=1$ of total order $2n-4$. Furthermore, the principal part at $z_i = 1$ is skew invariant under $z_i\mapsto \frac{1}{z_i}$ for each $i=1, 2, \ldots, n$. (Note that since there is only one pole, $W_0(x_1, \ldots, x_n)$ is equal to its principal part at $z_i = 1$ for each $i = 1, 2, \ldots, n$.)  Eynard and Orantin~\cite{EOrInv} show that there would also be poles at $z_i=-1$ of order $2n-4$, {\em under the assumption that $y$ is analytic at the zeros of $\dd x$}. However, that assumption does not hold here, since $y$ has a pole at $z=-1$.

Let $\vc_n$ be the vector space of meromorphic differentials of a single variable with a pole only at $z=-1$ of order at most $2n$, and (with principal part) skew invariant under $z\mapsto \frac{1}{z}$. It has dimension $n$, because a basis for this vector space can be obtained by taking the principal part at $z=-1$ of $s^{-2k}\,\dd s$, for $k=1, 2, \ldots, n$ and $s$ the local coordinate defined by $x=4+s^2$.

The expansion of any $\xi\in\vc_n$ at $x=\infty$ can be understood from the following fundamental expansion.
\[
\xi_1(z):=\sum_{k=0}^{\infty}\binom{2k}{k}x^{-k}=\sqrt{\frac{x}{x-4}}=\frac{1+z}{1-z}
\]
(A different choice of $\sqrt{\frac{x}{x-4}}$ replaces $z$ by $\frac{1}{z}$ on the right hand side.)
Since $\xi_1(\frac{1}{z})=-\xi_1(z)$, then $d\xi_1\in\vc_1$. Consider the operator 
\[
-x\frac{\dd}{\dd x}=\frac{z(z+1)}{1-z}\frac{\dd}{\dd z}.
\]
It preserves skew invariance under $z\mapsto \frac{1}{z}$, since $x(\frac{1}{z})=x(z)$. Any function with only poles at $z=-1$ has no new poles introduced, because $z=\infty$ remains a regular point. %Hence, $P$ gives rise to a map $\vc_n\to\vc_{n+1}$ by 
%\[
%\begin{array}{ccc}
%\vc_n&\to&\vc_{n+1}\\
%\xi&\mapsto&P\int\xi
%\end{array}
%\]
%where any primitive of $\xi$ will suffice.
Hence, for the meromorphic function
\[
\xi_n(z)=\left(-x\frac{\dd}{\dd x}\right)^{n-1}\xi_1(z)=\sum_{k=1}^\infty k^{n-1}\binom{2k}{k}x^{-k},
\]
we have $\dd\xi_n\in\vc_n$. In particular, the dimension $n$ vector space $\vc_n$ is spanned by differentials with expansion around $x=\infty$ given by $\displaystyle\sum_{k=1}^\infty p(k)\binom{2k}{k}x^{-k-1}\,\dd x$, where $p(k)$ is a polynomial of degree at most $n$ with $p(0)=0$.

The multidifferential
\[
W_0(x_1, \ldots, x_n) \, \dd x_1 \cdots \dd x_n=\sum_{\mu_1, \ldots, \mu_n = 1}^\infty U_0(\mu_1, \ldots, \mu_n)\prod_{i=1}^n x_i^{-\mu_i-1}\dd x_i
\]
is a linear combination of monomials in the single variable differentials, and the total order of the pole corresponds to the total degree of the polynomial. Hence, this proves the theorem in the genus 0 case.

For higher genus, Theorem~\ref{th:main} shows that the generating function 
\[
W_g(x_1, \ldots, x_n)=\sum_{\mu_1, \ldots, \mu_n = 1}^\infty U_g(\mu_1, \ldots, \mu_n)\prod_{i=1}^n x_i^{-\mu_i-1}
\]
is an expansion of a function which is rational in $z_i$ for $x_i = z_i + \frac{1}{z_i} + 2$, with poles only at $z_i=1$ of total order $6g+2n-4$ and at $z_i=-1$ of total order $2g$. Furthermore, the principal part at $z_i=\pm 1$ is skew invariant under $z_i\mapsto \frac{1}{z_i}$ for each $i=1, 2, \ldots, n$.

The operator 
\[
-\frac{\dd}{\dd x}=\frac{z^2}{1-z^2}\frac{\dd}{\dd z}
\]
introduces poles at $z=-1$. An order $2g-1$ pole at $z=-1$ is obtained from
\[
\left(-\frac{\dd}{\dd x}\right)^gp(k)\binom{2k}{k}x^{-k}= k(k+1) \cdots (k+g-1) \, p(k)\binom{2k}{k}x^{-k-g}=q(m)\binom{2m-2g}{m-g}x^{-m}.
\]
A pole of order $2g-1$ also consists of lower order terms, or equivalently, terms coming from lower genus contributions.  Given polynomials $p_i(m)$ for $i=0, 1, \ldots, g$, there exists a polynomial $p(m)$ such that
\[
p(m)\frac{(2m-2g)!}{m!(m-g)!}=p_g(m)\binom{2m-2g}{m-g}+p_{g-1}(m)\binom{2m-2g+2}{m-g+1}+ \cdots +p_0(m)\binom{2m}{m},
\]
which proves the theorem.
\end{proof}

\begin{proof}[Proof of Proposition~\ref{th:formg=0}.]
The proof uses the structure theorem and an elementary relation known as the divisor equation.  By Theorem~\ref{th:struc}
\[
B_{0,n}(\mu_1, \ldots, \mu_n)=p(\mu_1, \ldots, \mu_n)\prod_{i=1}^n\binom{2\mu_i}{\mu_i}, \qquad \text{for } n \geq 3.
\]
The following {\em divisor equation} is easy to prove combinatorially, by considering the result of doubling any edge of a dessin.
\begin{equation}   \label{divisor}
B_{0,n+1}(1,\mu_1, \ldots, \mu_n)=|\mmu|\cdot B_{0,n}(\mu_1, \ldots, \mu_n)
\end{equation}
In terms of the polynomial part of $B_{0,n}$, it implies that 
\begin{equation}   \label{div}
p(1,\mu_1, \ldots, \mu_n)=\frac{1}{2}|\mmu| \cdot p(\mu_1, \ldots, \mu_n).
\end{equation}
It is easy to check that the recursion \eqref{div} is satisfied by
\[
p(\mu_1, \ldots, \mu_n)=2^{1-n}(|\mmu|-1) (|\mmu|-2) \cdots (|\mmu|-n+3).
\]
A degree $n-3$ symmetric polynomial in $n$ variables $p(\mu_1, \ldots, \mu_n)$ is uniquely determined by its evaluation at one variable, say $\mu_1=a$. Hence, equation~\eqref{div} uniquely determines $p(\mu_1, \ldots, \mu_n,\mu_{n+1})$ from $p(\mu_1, \ldots, \mu_n)$ and inductively, from the initial condition $p(\mu_1,\mu_2,\mu_3) = \frac{1}{4}$. (This last fact is equivalent to $b_{0,3}(\mu_1, \mu_2, \mu_3) = 2$.) Thus, the required solution $p(\mu_1, \ldots, \mu_n)=2^{1-n}(|\mmu|-1) (|\mmu|-2) \cdots (|\mmu|-n+3)$ is indeed the unique solution and the proposition is proven.
\end{proof}

%\subsection{The dilaton equation and pointed bipartite graphs}  \label{sec:pointed}

The definition of $B_{g,n}(\mu_1, \ldots, \mu_n)$ requires all $\mu_i$ to be positive. However, the polynomial structure of $B_{g,n}(\mu_1, \ldots, \mu_n)$ allows one to evaluate at $\mu_i=0$. The following proposition gives a combinatorial meaning to such evaluation.

\begin{proposition}  \label{th:pointed}
For $m$ and $n$ positive integers, $B_{g,n+m}(\mu_1, \ldots, \mu_n, 0, 0, \ldots, 0)$ enumerates dessins in $\cb_{g,n}(\mu_1, \ldots, \mu_n)$ with $m$ distinct black vertices labelled.
\end{proposition}

\begin{proof}
Intuitively, $\mu_i=0$ represents a diameter 0 boundary component, which we think of as the labelled black vertices. Thus, it is natural to use the labels $n+1, n+2, \ldots, n+m$. More precisely, we can use Theorem~\ref{th:main} to write the dilaton equation \eqref{dilaton} equivalently as
\[
B_{g,n+1}(1,\mu_1, \ldots, \mu_n)-B_{g,n+1}(0,\mu_1, \ldots, \mu_n)=\frac{1}{2}(|\mmu|+2g-2+n) \, B_{g,n}(\mu_1, \ldots, \mu_n).
\]
Together with the divisor equation \eqref{divisor}, we have
\[
B_{g,n+1}(0,\mu_1, \ldots, \mu_n)=\frac{1}{2}(|\mmu|-2g+2-n) \, B_{g,n}(\mu_1, \ldots, \mu_n) = \frac{1}{2} V \cdot B_{g,n}(\mu_1, \ldots, \mu_n),
\]
where $V = |\mmu|-2g+2-n$ is the number of vertices of any dessin in the set $\cb_{g,n}(\mu_1, \ldots, \mu_n)$. So the expression $V \cdot B_{g,n}(\mu_1, \ldots, \mu_n)$ can be interpreted as the enumeration of dessins in $\cb_{g,n}(\mu_1, \ldots, \mu_n)$ with one vertex labelled. Due to the symmetry between black and white vertices, one can interpret the expression $\frac{1}{2} V \cdot B_{g,n}(\mu_1, \ldots, \mu_n)$ as the enumeration of dessins in $\cb_{g,n}(\mu_1, \ldots, \mu_n)$ with one black vertex labelled. Now apply this relation $m$ times to the expression $B_{g,n+m}(\mu_1, \ldots, \mu_n, 0, 0, \ldots, 0)$ to obtain the desired result.
\end{proof}

The previous proposition does not apply when $n = 0$, but we conjecture that
\[
B_{g,m}(0, \ldots, 0)=2^{1-m} \, \chi(\modm_{g,m}).
\]
This would be a consequence of the dilaton equation and the $m=1$ case $B_{g,1}(0)=\chi(\modm_{g,1})$. We have not yet proven that $B_{g,1}(0)=\chi(\modm_{g,1})$, although we have verified it numerically for small values of $g$. It should follow from the three-term recursion of the next section, using a method analogous to that of Harer and Zagier~\cite{HZaEul}.

%I cannot yet prove \[ B_{g,n}(0, \ldots, 0)=2^{1-n}\cdot\chi(\modm_{g,n}). \]

%The compactified count probably uses $B_{g,n}(0, \ldots, 0)$.

\section{Three-term recursion for dessins d'enfant}  \label{sec:3term}

% where is the term fatgraph introduced?

Harer and Zagier calculated the virtual Euler characteristics of moduli spaces of smooth curves via the enumeration of fatgraphs with one face~\cite{HZaEul}. They define $\epsilon_g(n)$ to be the number of ways to glue the edges of a $2n$-gon in pairs and obtain an orientable genus $g$ surface. Equivalently, $\epsilon_g(n)$ is equal to $2n$ multiplied by the number of genus $g$ fatgraphs with one face and $n$ edges, counted with the usual weight $\frac{1}{|\mathrm{Aut}~\Gamma|}$. Through the analysis of a Hermitian matrix integral, they arrive at the following three-term recursion for fatgraphs with one face.
\[
(n+1) \,\epsilon_g(n) = 2(2n-1) \, \epsilon_g(n-1) + (n-1) (2n-1) (2n-3) \, \epsilon_{g-1}(n-2)
\]
There are myriad combinatorial results concerning the enumeration of fatgraphs with one face, and many of these have analogues for the enumeration of dessins with one face. For example, the first equation below, which appears in \cite{HZaEul} with $C(n,z)$ in place of $F_n(z)$, gives a formula for the polynomial generating function of fatgraphs with one face and n edges. The second is known as Jackson's formula~\cite{JacSom} and gives the polynomial generating function for dessins with one face and $n$ edges. % check this citation
\begin{align} \label{eq:fn}
F_n(z) &= \sum_{g=0}^\infty \epsilon_g(n) \, z^{n+1-2g} = \frac{(2n)!}{2^n n!} \sum_{r=0}^n 2^r \binom{n}{r} \binom{z}{r+1} \\ \label{eq:gn}
G_n(z) &= \sum_{g=0}^\infty U_g(n) \, z^{n+1-2g} = n! \sum_{r,s=0}^{n-1} \binom{n-1}{r,s} \binom{z}{r+1} \binom{z}{s+1}
\end{align}
Here, we use the notation $\binom{n-1}{r,s} = \frac{(n-1)!}{r! \, s! \, (n-1-r-s)!}$ with the convention that if $r + s > n-1$, then the expression is equal to zero.

To the best of our knowledge, the following analogue of the Harer--Zagier three-term recursion for dessins does not appear in the literature. It will play an important role in our calculation of the 1-point invariants of the spectral curve $xy^2 = 1$.

\begin{theorem}  \label{thm:3term}
The following recursion holds for all $g \geq 0$ and $n \geq 1$, where we set $U_0(0) = 1$.
\begin{equation}  \label{3term}
(n+1) \, U_g(n) = 2(2n-1) \, U_g(n-1) + (n-1)^2 (n-2) \, U_{g-1}(n-2)
\end{equation}
\end{theorem}

\begin{proof}
We begin with the following observation of Bernardi and Chapuy~\cite[Theorem 5.3]{BChBij}.
\[
G_n(z) = (n-1)! \, n! \sum_{i+j=n-1} \frac{F_i(z)}{(2i)!} \frac{F_j(z)}{(2j)!}.
\]
To see this, we simply substitute the expressions from equation~\eqref{eq:fn} and~\eqref{eq:gn} on both sides.
\begin{align*}
n! \sum_{r,s=0}^{n-1} \binom{n-1}{r,s} \binom{z}{r+1} \binom{z}{s+1} &= (n-1)! \, n! \sum_{i+j=n-1} \frac{1}{2^i i!} \frac{1}{2^j j!} \sum_{r=0}^i 2^r \binom{i}{r} \binom{z}{r+1} \sum_{s=0}^j 2^s \binom{j}{s} \binom{z}{s+1} \\
&= (n-1)! \, n! \sum_{i+j=n-1} \frac{1}{i! \, j!} \sum_{r,s=0}^{n-1}  2^{r+s-n+1} \binom{i}{r} \binom{j}{s} \binom{z}{r+1} \binom{z}{s+1} \\
&= (n-1)! \, n! \sum_{r,s=0}^{n-1} \frac{1}{r! \, s!} \binom{z}{r+1} \binom{z}{s+1} \sum_{k=0}^{n-1-r-s} \frac{2^{r+s-n+1}}{k! \, (n-1-r-s-k)!}
\end{align*}
The two sides are equal since the inner summation on the right hand side simplifies to $\frac{1}{(n-1-r-s)!}$. It immediately follows that
\[
G_n(z)= n! \, (n-1)! \, [t^{n+1}] E(z,t)^2,
\]
where $E(z,t)$ is the generating function
\begin{align*}
E(z,t) &= \sum_{n = 0}^\infty \frac{F_n(z)}{(2n)!} \, t^{n+1}
=\sum_{n = 0}^\infty \sum_{k = 0}^\infty \frac{2^k}{2^nn!}\binom{n}{k}\binom{z}{k+1} \, t^{n+1}
= t \sum_{k = 0}^\infty \frac{2^k}{k!}\binom{z}{k+1}\sum_{n = k}^\infty \frac{(t/2)^n}{(n-k)!} \\
&= t \, e^{t/2}\sum_{k = 0}^\infty \binom{z}{k+1}\frac{t^k}{k!}.
\end{align*}
This expansion implies that $E(z, t) = z M_{-z, 1/2}(t)$, where $M$ denotes the Whittaker function. It follows that $E(z, t)$ satisfies the following second order differential equation, which one can verify directly from the expansion above.
\[
\frac{\partial^2}{\partial t^2}E(z,t) - \left( \frac{1}{4}+\frac{z}{t} \right) E(z,t)=0
\]
Further differentiation shows that its square satisfies the following third order differential equation.
\[
\frac{\partial^3}{\partial t^3} E(z,t)^2=\left( \frac{4z}{t} + 1 \right) \frac{\partial}{\partial t} E(z,t)^2 - \frac{2z}{t^2} \, E(z,t)^2
\]
By collecting terms in the $t$-expansion of both sides, we obtain
\[
(n+1)n(n-1) \, [t^{n+1}] E(z,t)^2 = 4n z \, [t^n] E(z,t)^2 + (n-1) \, [t^{n-1}] E(z,t)^2 - 2z \, [t^n] E(z,t)^2,
\]
which is equivalent to
\[
(n+1) \, G_n(z) = 2(2n-1)z \, G_{n-1}(z) + (n-1)^2(n-2) \, G_{n-2}(z).
\]
Now we use $U_g(n)=[z^{n+1-2g}] G_n(z)$ to deduce that
\[
(n+1)\,U_g(n)=2(2n-1)\,U_g(n-1)+(n-1)^2(n-2)\,U_{g-1}(n-2). \qedhere
\]
\end{proof}

The three-term recursion for enumeration of dessins with one face is equivalent to a recursion for the generating functions $F_{g,1}(x)$. From this recursion, one can extract the highest order coefficients of the poles at $z=-1$. This is enough to prove Theorem~\ref{th:1point}, which states that the 1-point invariants of the spectral curve $xy^2 = 1$ are given by
\[
\omega^g_1(z)=2^{1-8g} \, \frac{(2g)!^3}{g!^4(2g-1)} \, z^{-2g} \, \dd z.
\]

\begin{proof}[Proof of Theorem~\ref{th:1point}]
As in Section~\ref{subsec:loop}, we use the notation $W_g(x)$ to denote
\[
W_g(x) \, \dd x = \dd F_{g,1}(x)=\sum_{n = 0}^\infty U_g(n) \, x^{-n-1} \, \dd x.
\]
Then equation~\eqref{3term} is equivalent to the differential equation 
\[
\left[\frac{x-4}{x}\frac{\dd}{\dd x}-\frac{2}{x^2}\right] W_g(x)=\left[\frac{\dd^3}{\dd x^3}+\frac{4}{x}\frac{\dd^2}{\dd x^2}+\frac{2}{x^2}\frac{\dd}{\dd x}\right] W_{g-1}(x). 
\]
Write $\omega^g_1(z) = w_g(z) \, \dd z=W_g(x) \, \dd x$, where $w_g(z)$ is a rational function of $z$ with poles at $z=\pm 1$.  The differential equation above can be written in terms of $w_g(z)$ and $w_{g-1}(z)$. Extract the highest order terms of the principal part at $z=-1$ of the resulting differential equation to obtain
\[
\left[\frac{\dd}{\dd z}-1\right] w_g(z) = \frac{1}{16} \left[\frac{\dd^3}{\dd z^3}+\frac{2}{1+z}\frac{\dd^2}{\dd z^2}-\frac{1}{(1+z)^2}\frac{\dd}{\dd z}+\frac{1}{(1+z)^3}\right] w_{g-1}(z) + [\,\text{lower order terms}\,].
\]
By ``lower order terms'', we mean terms with lower order poles at $z=-1$. This becomes an exact differential equation for the 1-point invariants $\omega^g_1(z) = v_g(z) \, \dd z$ of the spectral curve $xy^2=1$. For $g \geq 1$, we have
\[
\frac{\dd}{\dd z}v_g(z)= \frac{1}{16}\left[\frac{\dd^3}{\dd z^3}+\frac{2}{z}\frac{\dd^2}{\dd z^2}-\frac{1}{z^2}\frac{\dd}{\dd z}+\frac{1}{z^3}\right] v_{g-1}(z),
\]
with the boundary condition $v_g(\infty)=0$. This comes from the fact that the differential $v_g(z) \, \dd z$ has no pole at $z=\infty$ for $g \geq 1$.  Given the initial condition $v_0(z)=-2$, the system has a unique solution. Since $v_0(z)$ is homogeneous in $z$, each $v_g(z)$ is homogeneous of degree $-2g$. If we write $v_g(z) = a_g \, z^{-2g}$, then the previous equation implies that
\[
-2g \, a_g = -\frac{1}{16}(2g-3)(2g-1)^2 \, a_{g-1}.
\]
From the initial condition $a_0 = -2$, we obtain the solution $a_g = 2^{1-8g} \, \frac{(2g)!^3}{g!^4(2g-1)}$. It follows that the 1-point invariants of the spectral curve $xy^2 = 1$ are given by
\[
\omega^g_1(z) = v_g(z) \, \dd z=2^{1-8g} \, \frac{(2g)!^3}{g!^4(2g-1)} \, z^{-2g} \, \dd z. \qedhere
\]
\end{proof}

\section{Relation to the Kazarian--Zograf spectral curve}

Kazarian and Zograf \cite{KZoVir} study a more refined version of our enueration of dessins d'enfant. They define $\cn_{k,\ell}(\mmu)$ to be the weighted sum of connected Belyi covers with $k$ points above 0, $\ell$ points above 1 and ramification prescribed by $\mmu=(\mu_1, \ldots, \mu_n)$ above infinity.  Hence, we have the relation
\[
B_{g,n}(\mmu)=\sum_{k+\ell=|\mmu|+2-2g-n} \cn_{k,\ell}(\mmu).
\]  
They prove that the generating function 
\[
W^{KZ}_g(s,u,v,x_1,\dots,x_n)=\sum_{\mmu} \sum_{k+\ell=|\mmu|+2-2g-n} \cn_{k,\ell}(\mu)\,s^{|\mmu|}u^kv^{\ell}x_1^{\mu_1} \cdots x_n^{\mu_n}
\]
%for $N_{k,\ell}(\mu)$ is rational with poles of orders $6g-4+2n$ at $z=\pm1$ and analytic elsewhere.  It 
satisfies topological recursion on the regular spectral curve $C$ given by
\begin{equation}   \label{specurveKZ}
xy^2+(-\tfrac{1}{s}+(u+v)x)y+uvx=0
\end{equation}
One expects the $u\to v$ limit to be related to the unrefined count --- however, the curve \eqref{specurveKZ} is ill-behaved in the limit since $\omega^g_n\to 0$.  This can be seen from 
\[
x=\frac{1}{s}\frac{y}{(y+u)\,(y+v)} \quad \Rightarrow \quad \dd x=\frac{1}{s}\frac{uv-y^2}{(y+u)^2\,(y+v)^2} \, \dd y,
\]
% \[ x=\frac{z^2-1}{s(\alpha z^2-\beta)},\quad y=\sqrt{uv}\frac{1-z}{1+z} \]
so in the $u\to v$ limit $\dd x$ has a single zero since one of the two zeros of $\dd x$ cancels with a pole.  We can choose a family of rational parametrisations that fix the poles and zeros of $\dd x$ and hence counteract the collision of zeros and poles.
\[
x=\frac{\sqrt{uv}}{s(u-v)^2}\Big(z+\frac{1}{z}\Big)+\frac{u+v}{s(u-v)^2} \qquad \qquad 
y=-\sqrt{uv}\,\frac{z\sqrt{u}+\sqrt{v}}{z\sqrt{v}+\sqrt{u}}=-u+\frac{u-v}{1+\sqrt{\frac{v}{u}}\,z}
\]
The meromorphic functions
\[
x'=\frac{1}{s(v-u)}\Big[\sqrt{uv}\Big(z+\frac{1}{z}\Big)+u+v\Big] \qquad \text{and} \qquad 
y'=\frac{\sqrt{\frac{v}{u}}\,z}{1+\sqrt{\frac{v}{u}}\,z}
\]
define a spectral curve $C'$ that yields equivalent invariants $\omega^g_n(C)=\omega^g_n(C')$. The invariants are preserved under the transformations $y\mapsto y+v$ and $(x, y) \mapsto ((v-u) x, \frac{y}{v-u})$, since $x$ and $y$ appear in the recursion only via the combination $y\,\dd x$.  The limit $u\to v$ still causes the $\omega^g_n$ to degenerate to zero, but now in a controlled way. In fact, we have
\[
\omega^g_n(C)=(v-u)^{2g-2+n} \, \omega^g_n(C''),
\]
where $C''$ is the spectral curve given by
\[
x''=\sqrt{uv}\Big(z+\frac{1}{z}\Big)+u+v \qquad \text{and} \qquad
y''=\frac{\sqrt{\frac{v}{u}}\,z}{1+\sqrt{\frac{v}{u}}\,z}.
\]
This resembles our original spectral curve~\eqref{specurve}, although $C''$ is regular for $u\neq v$ and becomes irregular only when $u=v$. We have not been able to prove Theorem~\ref{th:main} via this limit. The qualitative difference between regular and irregular curves may explain this.

%This should lead to another proof of Theorem~\ref{th:main} except one would need to prove the same scaling in $W_{g,n}$, and rectify the fact that the expansion in KZ is in $x$, not $x^{-1}$.

\iffalse %
\subsection{Symplectic invariants}

The spectral curve \eqref{specurveKZ} has a symmetry
\[
 (u,v)\mapsto (v,u)
 \]
hence its symplectic invariants $F_g(s,u,v)$ are symmetric in $u$ and $v$.  The following conjecture is true for $g=2$.

\begin{conjecture} 
\[
F_g(s,u,v)=\chi(\modm_g)\cdot s^{2g-2}\left(\left(1-\frac{u}{v}\right)^{2g-2}+\left(1-\frac{v}{u}\right)^{2g-2}-1\right).
\]
\end{conjecture}
\fi %

\section{The quantum curve}   \label{sec:qua}

Recall from Theorem~\ref{th:main} that the topological recursion applied to the spectral curve
\[
x = z + \frac{1}{z} + 2 \qquad \text{and} \qquad y = \frac{z}{1+z}
\]
produces the invariants
\[
\omega^g_n(z_1, \ldots, z_n) = \frac{\partial}{\partial x_1} \cdots \frac{\partial}{\partial x_n} F_{g,n}(x_1, \ldots, x_n) \, \dd x_1 \otimes \cdots \otimes \dd x_n, \qquad \text{for } 2g-2+n > 0.
\]

%In this case, we obtain natural generating functions for the enumeration of dessins of type $(g,n)$. The constants of integration may be defined by imposing the condition that
%\[
%\sum_{z_i \in \pi^{-1}(x_i)} F_{g,n}(z_1, \ldots, z_i, \ldots, z_n) = 0, \qquad \text{for } i = 1, 2, \ldots, n.
%\]

Here, the so-called \emph{free energies}
\[
F_{g,n}(x_1, \ldots, x_n) = \sum_{\mu_1, \ldots, \mu_n = 1}^\infty B_{g,n}(\mu_1, \ldots, \mu_n) \prod_{i=1}^n x_i^{-\mu_i}%, \qquad \text{for } (g,n) \neq (0,1).
\] % is there a sign missing somehwere?
are natural generating functions for the enumeration of dessins of type $(g,n)$. By exception, we modify $F_{0,1}$ by defining
\[
F_{0,1}(x_1) = -\log x_1 + \sum_{\mu_1 = 1}^\infty B_{0,1}(\mu_1) \, x_1^{-\mu_1}.
\]
The logarithmic term appearing in the definition of $F_{0,1}(x_1)$ is consistent with the fact that  $U_0(0) = 1$ --- for example, see Section~\ref{subsec:loop} and Theorem~\ref{thm:3term} --- and the fact that
\[
\frac{\partial}{\partial x_1} \cdots \frac{\partial}{\partial x_n} F_{g,n}(x_1, \ldots, x_n) = (-1)^n \sum_{\mu_1, \ldots, \mu_n} U_g(\mu_1, \ldots, \mu_n) \prod_{i=1}^n x_i^{-\mu_i-1}.
\]

From these, one defines the wave function as follows, which differs from the expression given in Section~\ref{sec:intro} due to the adjustment of $F_{0,1}$.
\[
Z(x, \h) = \exp \bigg[ \sum_{g=0}^\infty \sum_{n=1}^\infty \frac{\h^{2g-2+n}}{n!} \, F_{g,n}(x, x, \ldots, x) \bigg]
\]

%Note that the spectral curve given parametrically above can be described by the plane algebraic curve
%\[
%P(x, y) = xy^2 - xy + 1 = 0.
%\]

\begin{theorem} \label{thm:qcurve}
For $\x = x$ and $\y = -\h \frac{\partial}{\partial x}$, we have the equation
\[
\left[ \y \x \y -\x \y + 1 \right] Z(x, \h) = 0,
\]
which is the quantum curve corresponding to the spectral curve $xy^2 - xy + 1 = 0$.
\end{theorem}

In order to interpret Theorem~\ref{thm:qcurve}, we need to make precise what we mean by this equation, given that the $\h$-expansion of $Z(x, \h)$ is not well-defined. One way to do this is to express the wave function as
\[
Z(x, \h) = x^{-1/\h} \, \Z(x, \h),
\]
where the term $x^{-1/\h}$ comes from the exceptional logarithmic term in the definition of $F_{0,1}$. So we interpret Theorem~\ref{thm:qcurve} as
\begin{equation} \label{abc}
x^{1/\h} \left[ \y \x \y -\x \y + 1 \right] x^{-1/\h} \, \Z(x, \h) = 0 \qquad \Rightarrow \qquad \left[ x \h^2 \frac{\partial^2}{\partial x^2} + \h (\h - 2 + x) \frac{\partial}{\partial x} + x^{-1} \right] \Z(x, \h) = 0.
\end{equation}
The proposition below asserts that $\Z(x, \h)$ has an expansion in $x^{-1}$ with coefficients that are Laurent polynomials in $\h$~---~in other words, $\Z(x, \h) \in \mathbb{Q}[\h^{\pm 1}][[x^{-1}]]$. So the rigorous statement of Theorem~\ref{thm:qcurve} is via equation~\eqref{abc}, in terms of a differential operator annihilating the formal series $\Z(x, \h) \in \mathbb{Q}[\h^{\pm 1}][[x^{-1}]]$.

In fact, we will explicitly calculate the coefficients of $\Z(x, \h)$ in the $x^{-1}$-expansion and use this to derive the quantum curve. The strategy is to interpret the coefficients of $\Z(x, \h)$ combinatorially using the following observations about the definition of the wave function $Z(x, \h)$.
\begin{itemize}
\item The expression $F_{g,n}(x, x, \ldots, x)$ counts dessins not with respect to the tuple of boundary lengths, but with respect to the sum of the boundary lengths. This is precisely the number of edges in the dessin.
\item The term $\frac{\h^{2g-2+n}}{n!}$ ignores the labels of the boundary components and organises the count of dessins by Euler characteristic rather than by genus.
\item The exponential in the definition of $Z(x, \h)$ passes from a count of connected dessins to a count of disconnected dessins, via the usual exponential formula.
\end{itemize}

\begin{proposition} \label{pro:wfunction}
The modified wave function $\Z(x, \h)$ is an element of $\mathbb{Q}[\h^{\pm 1}][[x^{-1}]]$. Furthermore, it can be expressed as
\[
\Z(x, \h) = 1 + \sum_{e=1}^\infty \h^e \left[ \h^{-1} (\h^{-1} + 1) (\h^{-1} + 2) \cdots (\h^{-1} + e-1) \right]^2 x^{-e}.
\]
\end{proposition}

\begin{proof}
First, consider the logarithm of the modified wave function.
\begin{align*}
\log \Z(x, \h) &= \sum_{g=0}^\infty \sum_{n=1}^\infty \frac{\h^{2g-2+n}}{n!} \, F_{g,n}(x, x, \ldots, x) \\
&= \sum_{g=0}^\infty \sum_{n=1}^\infty \frac{\h^{2g-2+n}}{n!} \sum_{\mu_1, \ldots, \mu_n = 1}^\infty B_{g,n}(\mu_1, \ldots, \mu_n) \, x^{-(\mu_1 + \cdots + \mu_n)} \\
&= \sum_{v=1}^\infty \sum_{e=1}^\infty f(v,e) \, \h^{e-v} x^{-e}
\end{align*}
Here, $f(v, e)$ denotes the weighted count of connected dessins with $v$ vertices, $e$ edges, and unlabelled boundary components. To obtain this last expression, we have used the fact that $v - e = 2g - 2 + n$ and $\mu_1 + \cdots + \mu_n = e$ for any dessin. The factor $\frac{1}{n!}$ accounts for the fact that we are now considering dessins with unlabelled faces. Note that we exclude from consideration the dessins consisting of an isolated vertex.

Next, we use the exponential formula to pass from the connected count to its disconnected analogue.
\[
\Z(x, \h) = 1 + \sum_{v=1}^\infty \sum_{e=1}^\infty f^\bullet(v, e) \, \h^{e-v} x^{-e}
\]
Here, $f^\bullet(v, e)$ denotes the weighted count of possibly disconnected dessins with $v$ vertices, $e$ edges, and unlabelled faces. We furthermore require that no connected component consists of an isolated vertex.

% we haven't defined dessins in terms of permutations yet

Now note that $f^\bullet(v, e)$ is equal to $\frac{1}{e!}$ multiplied by the number of triples $(\sigma_0, \sigma_1, \sigma_2)$ of permutations in the symmetric group $S_e$ such that $\sigma_0 \sigma_1 \sigma_2 = \text{id}$ and $c(\sigma_0) + c(\sigma_1) = v$. Here, we use $c(\sigma)$ to denote the number of disjoint cycles in the permutation $\sigma$. However, this is clearly equal to $\frac{1}{e!}$ multiplied by the number of pairs $(\sigma_0, \sigma_1)$ of permutations in $S_e$ such that $c(\sigma_0) + c(\sigma_1) = v$. Recall that the Stirling number of the first kind $\stirling{n}{k}$ counts the number of permutations in $S_n$ with $k$ disjoint cycles. So we have deduced that\footnote{The numbers $f^\bullet(v, e)$ appear in the triangle of numbers given by sequence A246117 in the OEIS. There, the number $f^\bullet(v, e)$ is described as the number of parity-preserving permutations in $S_{2e}$ with $v$ cycles. A parity-preserving permutation $p$ on the set $\{1, 2, \ldots, n\}$ is one that satisfies $p(i) \equiv i \pmod{2}$ for $i = 1, 2, \ldots, n$.}%This interpretation is evident from the formula.}
\[
f^\bullet(v, e) = \frac{1}{e!} \sum_{a + b = v} \stirling{e}{a} \stirling{e}{b}.
\]
It is evident from this formula that for fixed $e$, we require $2 \leq v \leq 2e$ to have $f^\bullet(v, e) \neq 0$. Therefore, the modified wave function $\Z(x, \h)$  is indeed an element of $\mathbb{Q}[\h^{\pm 1}][[x^{-1}]]$.

Now we simply use the fact that the generating function for Stirling numbers of the first kind is given by
\[
\sum_{k=1}^n \stirling{n}{k} x^k = x (x+1) (x+2) \cdots (x+n-1).
\]
Use this in the expression for the modified wave function as follows.
\begin{align*}
\Z(x, \h) &= 1 + \sum_{v=1}^\infty \sum_{e=1}^\infty \frac{1}{e!} \sum_{a + b = v} \stirling{e}{a} \stirling{e}{b} \h^{e-v} x^{-e} \\
&= 1 + \sum_{e=1}^\infty \frac{1}{e!} \sum_{a=1}^\infty \stirling{e}{a} \h^{-a} \sum_{b=1}^\infty \stirling{e}{b} \h^{-b} \, \h^e x^{-e} \\
&= 1 + \sum_{e=1}^\infty \frac{\h^e}{e!} \left[ \h^{-1} (\h^{-1} + 1) (\h^{-1} + 2) \cdots (\h^{-1} + e-1) \right]^2 x^{-e} \qedhere
\end{align*}
\end{proof}

The quantum curve for the enumeration of dessins is now a straightforward consequence of the previous proposition.

\begin{proof}[Proof of Theorem~\ref{thm:qcurve}]
We use Proposition~\ref{pro:wfunction} to derive the quantum curve. Start by writing
\[
\Z(x, \h) = \sum_{e=0}^\infty a_e(\h) \, x^{-e}, \qquad \text{where } a_e(\h) = \frac{\h^e}{e!} \left[ \h^{-1} (\h^{-1} + 1) (\h^{-1} + 2) \cdots (\h^{-1} + e-1) \right]^2.
\]
Then take the relation $(e+1) \, a_{e+1}(\h) = \h (\h^{-1} + e)^2 \, a_e(\h)$, multiply both sides by $x^{-e-1}$, and sum over all $e$.
\begin{align*}
\sum_{e=0}^\infty (e+1) \, a_{e+1}(\h) \, x^{-e-1} &= \sum_{e=0}^\infty \h (\h^{-1} + e)^2 \, a_e(\h) \, x^{-e-1}\\
\sum_{e=0}^\infty e \, a_e(\h) \, x^{-e} &= \h^{-1} \sum_{e=0}^\infty a_e(\h) \, x^{-e-1} + 2 \sum_{e=0}^\infty e \, a_e(\h) \, x^{-e-1} + \h \sum_{e=0}^\infty e^2 \, a_e(\h) \, x^{-e-1} \\
-x \frac{\partial \Z}{\partial x} &= \h^{-1} x^{-1} \Z - 2 \frac{\partial \Z}{\partial x} + \h \frac{\partial}{\partial x} \left[ x \frac{\partial \Z}{\partial x} \right]
\end{align*}
Now use the product rule on the final term and rearrange the equation to obtain the desired quantum curve, as expressed in equation~\eqref{abc}.
\end{proof}

\begin{remark}
In the semi-classical limit, the quantum curve differential operator becomes a multiplication operator. The limit is obtained by sending $\h \to 0$ in the following way. Put
\[
S_m(x)=\displaystyle\sum_{2g-2+n=m-1} \frac{1}{n!} \, F_{g,n}(x, x, \ldots, x) \qquad \Rightarrow \qquad Z(x, \h) = \exp \bigg[ \sum_{m=0}^\infty \h^{m-1} \, S_m(x) \bigg].
\]
Then we have
\begin{align*}
&\, \lim_{\hbar\to 0} \exp \bigg[-\frac{1}{\hbar}S_0(x)\bigg] (\y \x \y - \x \y + 1) \, Z(x, \h) \\
=&\, \lim_{\hbar\to 0} \exp \bigg[-\frac{1}{\hbar}S_0(x)\bigg] (\y \x \y - \x \y + 1) \exp\bigg[\frac{1}{\hbar}S_0(x)\bigg] \exp\bigg[\sum_{m=1}^\infty \hbar^{m-1} S_m(x)\bigg] \\
=&\, \lim_{\hbar\to 0} \left( xS_0'(x)^2-xS_0'(x)+1+\hbar S_0'(x)-\hbar \right) \exp\bigg[\sum_{m=1}^\infty \hbar^{m-1} S_m(x)\bigg].
\end{align*}
For this expression to vanish, we must have $xS_0'(x)^2-xS_0'(x)+1=0$, which is precisely the spectral curve given by equation~\eqref{specurve} since $y=F_{0,1}'(x)=S_0'(x)$.
\end{remark}

\begin{remark}
In other rigorously known instances of the topological recursion/quantum curve paradigm where the spectral curve is polynomial, the quantum curve is often obtained using the normal ordering of operators that places differentiation operators to the right of multiplication operators. For example, 
\begin{itemize}
\item the spectral curve $y^2 - xy + 1 = 0$ that governs the enumeration of ribbon graphs has quantum curve $\y^2 - \x \y + 1$~\cite{MSuSpe};
\item the spectral curve $xy^2 + y + 1 = 0$ that governs monotone Hurwitz numbers has quantum curve $\x \y^2 + \y + 1$~\cite{DDMTop}; and
\item the spectral curve $y^a - xy + 1 = 0$ that governs the enumeration of $a$-hypermaps has quantum curve $\y^a - \x \y + 1$~\cite{DMaQua,DOPSCom}.
\end{itemize}
In this particular instance, we have a quantum curve that is not obtained simply by the normal ordering of operators. Imposing a normal ordering introduces an $\h$ term in the following way, where we use the commutation relation $[\x, \y] = \h$.
\[
\widehat{P}(\x, \y) = \y \x \y -\x \y  + 1 = (\x \y - \h) \y - \x \y + 1 = (\x \y^2 - \x \y + 1) - \h \y
\]
\end{remark}

\begin{remark}
The statement
\[
\lim_{\hbar\to 0} \hbar\frac{\dd}{\dd x}\log \Z(x,\hbar)= y = \int_0^4\frac{\lambda(t)}{x-t}\,\dd t,\qquad \text{where } \lambda(t)=\frac{1}{2\pi}\sqrt{\frac{4-t}{t}}\cdot\mathbbm{1}_{[0,4]}
\]
agrees with equation (3.4) of \cite{FLiRan}, where $\Z(x,\hbar)$ is replaced by the expectation $\langle\det(x-A)\rangle$ of a matrix integral over positive definite Hermitian matrices. This confirms the known fact in the physics literature that the wave function corresponds to the expectation $\langle\det(x-A)\rangle$.
% $-\frac{1}{x}+y=y^2$ ... relevant?
\end{remark}

% mention the ambiguities involved in finding the quantum curve (again)... maybe they can be used to get rid of the linear term in h?

\section{Local irregular behaviour}  \label{sec:asym}

The asymptotic behaviour of $\omega^g_n$ near zeros of $\dd x$ is governed by the local behaviour of the curve $C$ there~\cite{EOrTop}. The usual assumption is that the local behaviour is described by $x=y^2$ which, as a global curve, has invariants $\omega^g_n$ that store tautological intersection numbers over the compactified moduli space of curves $\overline{\modm}_{g,n}$. Here, we also consider the local behaviour described by $xy^2=1$.

Consider the rational spectral curve
\begin{equation}  \label{air2}
x=\frac{1}{2}z^2 \qquad \text{and} \qquad y=\frac{1}{z}.
\end{equation}
We include the factor of $\frac{1}{2}$ in $x$ simply to reduce powers of 2 in the resulting invariants. One can calculate invariants via topological recursion and obtain
\begin{align*}
\omega^0_n&=0, \qquad \text{for } n \geq 3 \\
\omega^1_n &= 2^{-3}(n-1)!\prod_{i=1}^n\frac{\dd z_i}{z_i^2} \\
\omega^2_n&=2^{-8}3^2(n+1)!\prod_{i=1}^n\frac{\dd z_i}{z_i^2}\sum_{i=1}^n\frac{1}{z_i^2}.
\end{align*}

If we write
\[
\omega^g_n=\sum u_g(\mu_1, \ldots, \mu_n) \prod_{i=1}^n \frac{\dd z_i}{z_i^{\mu_i+1}},
\]
then the coefficients satisfy the recursion
\begin{equation}  \label{recasy}
u_g(\mu_1,\mmu_S)=\sum_{j=2}^n\mu_j \,u_g(\mu_1+\mu_j-1,\mmu_{S\setminus\{j\}})
+\frac{1}{2} \sum_{i+j=\mu_1-1} \bigg[u_{g-1}(i,j,\mmu_S) + \mathop{\sum_{g_1+g_2=g}}_{I \sqcup J = S} u_{g_1}(i,\mmu_I)\,u_{g_2}(j,\mmu_J)\bigg],
\end{equation}
for $S=\{2, \ldots, n\}$. We impose the base cases $u_0(\mu_1, \ldots, \mu_n) = 0$ for all $\mu_1, \ldots, \mu_n$ and $u_g(\mu_1, \ldots, \mu_n) = 0$ if any of $\mu_1, \ldots, \mu_n$ are even. In low genus, the recursion is solved by
\begin{align*}
u_1(1, \ldots, 1)&=2^{-3}(n-1)! &\text{and}&& u_1(\mu_1, \ldots, \mu_n)&=0\quad\text{otherwise,} \\
u_2(3,1, \ldots, 1)&=2^{-8}3(n+1)! &\text{and}&& u_2(\mu_1, \ldots, \mu_n)&=0\quad\text{otherwise,} \\
u_3(5,1, \ldots, 1)&=2^{-13}75(n+3)! & \\
u_3(3,3,1, \ldots, 1)&=2^{-12}\frac{189}{5}(n+3)! &\text{and}&& u_3(\mu_1, \ldots, \mu_n)&=0\quad\text{otherwise.}
\end{align*}

The invariant $u_g(\mu_1, \ldots, \mu_n)$ is non-zero only if $\mmu$ is a partition of $2g-2+n$ with  only odd parts. This suggests a possible relationship with connected branched covers of the torus with $n$ branch points of ramification orders $\mu_1, \ldots, \mu_n$. By the Riemann--Hurwitz formula, such a cover is necessarily of genus $g$.

\subsection{Volumes}  
One can associate polynomials $V_g(L_1, \ldots, L_n)$ to the curve \eqref{air2}, which are dual to the ancestor invariants $u_g(\mu_1, \ldots, \mu_n)$. We refer to them as volumes, since they have properties that resemble the Kontsevich volumes associated to the cell decomposition of the moduli space of curves~\cite{KonInt}. These polynomials satisfy
\[
\cl\left[V_g(L_1, \ldots, L_n)\right]=\int_0^{\infty} \!\cdots\! \int_0^{\infty}V_g(L_1, \ldots, L_n)\prod \exp(-z_iL_i) \cdot L_i \, \dd L_i=\omega^g_n(z_1, \ldots, z_n).
\]
Note that $\cl(L^{2k})=\dfrac{(2k+1)!}{z^{2k+2}}$.

We highlight several properties of these volumes.
\begin{enumerate}
\item For $S=\{2, \ldots, n\}$, we have the recursion
\begin{align*}
2L_1V_g(L_1,\LL_S) &=\sum_{j=2}^n \bigg[ (L_j+L_1) V_g(L_j+L_1, \LL_{S\setminus\{j\}}) - (L_j-L_1) V_g(L_j-L_1, \LL_{S\setminus\{j\}}) \bigg] \\
&+\int_0^{L_1}\dd x\cdot x(L_1-x)\bigg[
V_{g-1}(x,L_1-x,\LL_S)+\mathop{\sum_{g_1+g_2=g}}_{I\sqcup J=S}V_{g_1}(x,\LL_I) \, V_{g_2}(L_1-x,\LL_J)\bigg].
\end{align*}
\item The volume $V_g(L_1, \ldots, L_n)$ is a degree $2g-2$ polynomial in $L_1, \ldots, L_n$.
\item The volume $V_g$ depends on $n$ in a mild way --- we have
\[
V_g(L_1, \ldots, L_n)=(2g-3+n)!\sum_{\mmu\vdash g-1} C_g(\mmu) \, m_{\mmu}(\LL^2),
\] % is c_g(\mu) correct here?
where the summation is over partitions $\mmu$ of $g-1$, the expression $m_{\mmu}(\LL^2)$ denotes the monomial symmetric function in $L_1^2, \ldots, L_n^2$, and $C_g(\mmu)$ are constants.
\item There exists the following dilaton equation for the volumes.
\[
 V_g(L_1, \ldots, L_n,0)=(2g-2+n)\, V_g(L_1, \ldots, L_n)
 \]
\item One can calculate the following formulae.
\begin{align*}
V_1(L)&=2^{-3} \cdot (n-1)!\\
V_2(L)&=2^{-9}\cdot3\cdot (n+1)!\sum L_i^2\\
V_3(L)&=2^{-16}\cdot (n+3)! \left( 5\sum L_i^4+\frac{84}{5}\sum L_i^2L_j^2 \right) \\
\text{\quad}V_g(L)&=2^{2-6g}\binom{2g}{g}\frac{(2g-3+n)!}{(g-1)!^2}\sum L_i^{2g-2}+ \cdots 
\end{align*}
\end{enumerate}
The recursion may help to answer the question: volumes of {\em what}?

\appendix

\section{Formulae.}

In the following table, we use the notation introduced earlier.
\[
c_g(\mu)=\frac{(2\mu-2g)!}{\mu! \, (\mu-g)!}=\binom{2\mu}{\mu}2^{-g}\prod_{k=1}^g\frac{1}{2\mu-2k+1}
\]

\begin{center}
\begin{tabular}{ccl} \toprule
$g$ & $n$ & $\dfrac{B_{g,n}(\mu_1, \ldots, \mu_n)}{\prod c_g(\mu_i)}$ \\ \midrule
0 & 1 & $\frac{1}{\mu_1(\mu_1+1)}$ \\
0 & 2 & $\frac{1}{2(\mu_1+\mu_2)} $ \\
0 & 3 & $\frac{1}{4} $ \\
0 & $n$ & $2^{1-n} (|\mmu|-1) (|\mmu|-2) \cdots (|\mmu|-n+3)$ \\
1 & 1 & $\frac{1}{12} (\mu_1-1) (\mu_1-2)$ \\
1 & 2 & $\frac{1}{12}(2\mu_1^2\mu_2^2+2\mu_1^3\mu_2+2\mu_1\mu_2^3-\mu_1^3-\mu_2^3-9\mu_1^2\mu_2-9\mu_1\mu_2^2+4\mu_1^2+4\mu_2^2+14\mu_1\mu_2-5\mu_1-5\mu_2+2)$ \\
2 & 1 & $\frac{1}{1440} (\mu_1-1) (\mu_1-2) (\mu_1-3) (\mu_1-4) (5\mu_1^2-7\mu_1+6)$ \\
%2 & 2 & $\frac{1}{720}(\mu_1+\mu_2-3)\cdot[\,\text{degree 8 polynomial in $\mu_1$ and $\mu_2$}\,]$ \\
3 & 1 & $\frac{1}{362880} (\mu_1-1) (\mu_1-2) (\mu_1-3) (\mu_1-4) (\mu_1-5) (\mu_1-6) (35\mu_1^4-182\mu_1^3+397\mu_1^2-346\mu_1+240)$ \\ \bottomrule
\end{tabular}
\end{center}

%I tried to write things as Catalan-type combinatorial factor $\times$ polynomial of predictable degree, in this case being $3g-3+2n$.  This failed because I would need $B_{2,1}$ to have a factor of $\frac{1}{b-1}\binom{2b-4}{b-2}$.

\end{document}